\documentclass[10.05pt,a4paper]{amsart}
\usepackage{rotating}
\usepackage[all]{xy}

\usepackage[utf8]{inputenc}
\usepackage{amscd,amssymb,amsopn,amsmath,amsthm,graphics,amsfonts,enumerate,verbatim,calc}
\headsep=1cm \numberwithin{equation}{section}
\newtheorem{theorem}{Theorem}[section]
\newtheorem{lemma}[theorem]{Lemma}
\newtheorem{proposition}[theorem]{Proposition}
\newtheorem{corollary}[theorem]{Corollary}
\newtheorem{notation}[theorem]{Notation}
\newtheorem{observation}[theorem]{Observation}
\theoremstyle{definition}
\newtheorem{discussion}[theorem]{Discussion}
\newtheorem{definition}[theorem]{Definition}
\theoremstyle{remark}
\newtheorem{remark}[theorem]{Remark}\newtheorem{fact}[theorem]{Fact}
\newtheorem{example}[theorem]{Example}
\newtheorem{question}[theorem]{Question}
\newtheorem{claim}[theorem]{Claim}
\newtheorem{hypothesis}[theorem]{Hypothesis}

\newcommand{\Ass}{\operatorname{Ass}}

\newcommand{\grade}{\operatorname{grade}}
\newcommand{\pgrade}{\operatorname{p.grade}}
\newcommand{\Kgrade}{\operatorname{K.grade}}

\newcommand{\Min}{\operatorname{Min}}

\newcommand{\Spec}{\operatorname{Spec}}

\newcommand{\h}{\operatorname{h}}
\newcommand{\Gpd}{\operatorname{Gpd}}

\newcommand{\End}{\operatorname{End}}
\newcommand{\Ht}{\operatorname{ht}}
\newcommand{\id}{\operatorname{id}}
\newcommand{\fd}{\operatorname{fd}}

\newcommand{\pd}{\operatorname{pd}}
\newcommand{\Gdim}{\operatorname{Gdim}}
\newcommand{\gdim}{\operatorname{gd}}

\newcommand{\E}{\operatorname{E}}

\newcommand{\Ext}{\operatorname{Ext}}
\newcommand{\Supp}{\operatorname{Supp}}

\newcommand{\Tor}{\operatorname{Tor}}
\newcommand{\Hom}{\operatorname{Hom}}

\newcommand{\Zd}{\operatorname{Zd}}
\newcommand{\Ann}{\operatorname{Ann}}

\newcommand{\wdim}{\operatorname{w.dim}}

\newcommand{\depth}{\operatorname{depth}}
\newcommand{\pdepth}{\operatorname{p.depth}}

\newcommand{\Ker}{\operatorname{Ker}}
\newcommand{\Coker}{\operatorname{Coker}}
\newcommand{\im}{\operatorname{im}}

\newcommand{\vpl}{\operatornamewithlimits{\varprojlim}}

\newcommand{\lo}{\longrightarrow}
\newcommand{\fm}{\frak{m}}
\newcommand{\fp}{\frak{p}}

\newcommand{\fa}{\frak{a}}
\newcommand{\fb}{\frak{b}}

\newcommand{\fn}{\frak{n}}

\newcommand{\gldim}{\operatorname{gldim}}
\newcommand{\Frac}{\operatorname{Frac}}
\newcommand{\trdeg}{\operatorname{trdeg}}
\newcommand{\supp}{\operatorname{supp}}

\newcommand{\m}{\mathfrak m}

\newenvironment{PROOF}[2][\proofname.]
{\begin{proof}[#1]\renewcommand{\qedsymbol}{\textsquare$_{\rm #2}$}}
	{\end{proof}}

\begin{document}

\author[M. Asgharzadeh  and E. Mahdavi]{Mohsen Asgharzadeh and Elham Mahdavi}

\title[Homology with the theme of Matlis]{Homology with the theme of Matlis}

\address{M. Asgharzadeh }
\email{mohsenasgharzadeh@gmail.com}

\address{E. Mahdavi}
\email{elham.mahdavi.gh@gmail.com}

\subjclass[2020]{13C14; 13D02; 13D09}

\keywords{Cohen-Macaulay rings; weakly cotorsion; completion; fraction field; projective dimension; Ext-modules; $R$-topology}

\begin{abstract}  Matlis proved a lot of homological properties of the fraction field of an integral domain. In this paper, we simplify and extend some of
	them from 1-dimensional (resp. rank one) cases to the higher dimensional (resp. finite rank) cases. For example, we study the weakly cotorsion property of $\Ext(-,\sim)$, and use it to present splitting criteria.
	These are equipped with several applications.
	For instance, we compute the projective dimension of $\widehat{R}$ and present some non-noetherian versions of Grothendieck's localization problem. We construct a new class of co-Hopfian modules and extend Matlis' decomposability problem to higher ranks. In particular, this paper deals with the basic properties of Matlis' quadratic  $(Q,Q/R, \widehat{R},\overset{\sim}R).$ 
\end{abstract}

\maketitle
\tableofcontents
\section{Introduction}

Matlis has a lot of contributions to commutative algebra. His results on decomposition of injective modules and also his duality-theory are very famous, and perhaps, almost
every person in the community is family with his works. Also, his contributions to understanding $\pd_R(Q)$ is very important. Despite these,
some of his results are less well-known. For instance, some of his results reproved  without any contribution to him.

In \cite[10.4]{mat2}, and over 1-dimensional domains, Matlis proved that $\Ext^{>1}_R(-,\sim)$ is weakly co-torsion. A natural question arises. In fact,
Fuchs and Salce  \cite[page 462, problem 48]{FS} asked:
\begin{question}
  Is $\Ext_R^{1}(M,N)$ weakly cotorsion?\end{question}

In Section 2, and in a series of cases, we answer this. First, we observe that Question 1.1 is true if both modules are finitely generated. Then we try to remove these finiteness conditions. 
Here are the two main answers established below:
\begin{theorem}\begin{enumerate}
	\item For an \emph{arbitrary commutative domain} and \emph{any finitely
		generated} $M$, with $N$ arbitrary, there is $0\ne r\in R$ such that
	$$
	r\Ext_R^q(M,N)=0\quad(q\geq1),\qquad
	\Ext_R^p\bigl(Q,\Ext_R^q(M,N)\bigr)=0\quad(p\geq0,\ q\geq1).
	$$
	The same $r$ works for every $N$ and every positive $q$.
	\item If $\pd_R Q\leq1$, then, for \emph{arbitrary} $M,N$,
$
	\Ext_R^p\bigl(Q,\Ext_R^q(M,N)\bigr)=0$ for $q,p \geq 1.
	$
	
\end{enumerate}
\end{theorem}
Domains satisfying $\pd_R Q\leq1$ are called Matlis domains. Then we reprove one of Matlis results about the weakly cotorsion property of $\Ext_R^{1}(M,N)$.
As an application to  the weakly co-torsion results,  we present the following splitting criteria in  Section 3:

\begin{corollary}
	Let $(R,\frak{m})$ be a 1-dimensional Gorenstein complete local integral domain and $M$ be a finitely generated $R$-module of finite injective dimension. Then $M \cong t(M) \oplus \overline{M}$ provided $\id(t(M))$ is finite.
\end{corollary}

 By $Q:=Q(R)$ we mean the fraction field of an integral domain $R$.
 In Section 4 we
study some homological properties of $Q$. Our results extend Matlis results to higher rank and higher dimensional cases. As a sample we show

\begin{corollary}
	Let $(R, \fm )$ be a complete local integral domain. The following assertions are equivalent:\begin{enumerate}
		\item[i)]Any torsion-free module of finite rank   is $Q^t\oplus M$ where $M$ is finitely generated.
		\item[ii)]  $\dim R\leq 1$.
		\item[iii)]Any torsion-free module of finite rank  is Matlis reflexive.	\item[iv)] Suppose $S$ is torsion-free and of finite rank. Then $\fm S=S$ iff $S$
		is injective.
	\end{enumerate}
\end{corollary}
Recall that Matlis was successful to study $Q/R$, at least over 1-dimensional Cohen-Macaulay rings. This is a very funny object. An essential  number  of his results,  may extend to the higher dimensional cases. In  this direction,  and for a sample, see \cite[Section 12]{moh}. We may come back to this.
Let us now,  use some local cohomological arguments to present the higher dimensional  version of \cite[Theorem 1]{matcanada}:
\begin{proposition}\label{31}  Let $(R,\fm)$ be a local  integral domain. Then $\dim Q(R)/R =\dim R -1$. In particular,
	$Q/R$ is an artinian if and only if $\dim(R) =1$.
\end{proposition}
Recall that a module
$A$  is called
\emph{co-Hopfian} if its injective endomorphisms are automorphisms. Even in group theory, it is difficult to find them. Matlis proved over 1-dimensional
Cohen-Macaulay rings that artinian modules are co-Hopfian. We drop both assumptions, namely Cohen-Macaulayness and 
the 1-dimensional assumption. We do this by reproving a funny result of Vasconcelos \cite{vo}. For more details, see Section 5. Suppose $(R,\fm)$ is a 1-dimensional domain.
In \cite[Theorem 7.1]{E} Matlis proved that
any nonzero map $Q/R\to Q/R$ is surjective iff $R$ is analytically unramified. We   present  its  higher-dimensional  version:
\begin{proposition} 
	Let $(R,\fm)$ be a local d-dimensional Cohen-Macaulay ring. Then any nonzero $f:H^d_\fm(R)\to H^d_\fm(R)$ is surjective iff $R$ is analytically irreducible.
\end{proposition}

Then, we use this generalization to present 
the following  higher dimensional version of \cite[7.1]{E}, where Matlis worked with 1-dimensional local domains:

\begin{corollary}
	Let $(R,\fm)$ be a d-dimensional complete  local Cohen-Macaulay domain.
	Then $H^d_\fm(R)$ has no nontrivial highly divisible submodule.
\end{corollary}

In \cite[Cor 8]{matcanada} Matlis proved over 1-dimensional domain that
$\Hom_R(Q/B,Q/R)\neq0$ where $B\subsetneqq Q$.
In Theorem \ref{fff} we extend this to higher.
We close Section 5 by presenting a sample that a  result of Matlis could not be extended to higher, see Theorem \ref{nex}.

Matlis invented a homological approach to the $R$-adic completion of $R$, denoted by $\overset{\sim} R$. Since the map from $R\to \overset{\sim} R$ is flat it is natural to ascent  and descent properties of $R$ and $\overset{\sim} R$. Since his rings where not assumed to be noetherian, the general Grothendieck's localization theory does not work directly.
In \cite{mat2} Matlis proved $\gldim(R) \leq \gldim(\overset{\sim} R)$, where $\gldim(-)$    is the global  dimension of $(-)$. This may not be sharp. He introduced the concept of closed rings and showed that in this case the equality holds (see \cite [Thm 4.6]{d}), in particular $R$ is regular if and only if $\overset{\sim} R$ is regular. 	It may be nice to mention that the class of  closed rings  appears naturally, and plays an essential role in many of Matlis papers.
Recently, the concept of Cohen-Macaulay rings extended to non-noetherian situation. For instance, see \cite{AT}. In Section 6 we present the Cohen-Macaulay  analogue of  \cite{mat2} and \cite{d}.
 Here is the Cohen-Macaulay version:
\begin{proposition}
	Let $(R,\fm)$ be a closed quasi-local domain. Then $R$ is Cohen-Macaulay (resp. Gorenstein) if and only if $\overset{\sim} R$ is Cohen Macaulay (resp. Gorenstein).
\end{proposition}

By   $\wdim (-)$ we mean the   weak  dimension of $(-)$.
Matlis \cite[9.2]{mat2}  proved that
 $\wdim (R) \leq \wdim( \overset{\sim} R).$ 
We close Section 6 by sharpening this.

Computing the projective dimension of   flat modules is a difficult task. The history of this comes  back to Kaplansky's school. As a sample, see Bass' paper \cite{bass}. In fact, Matlis was interested in computing $\pd_R(Q)$. Also, see \cite{moh}. Section 7 deals with the following problem:

\begin{question}
What is $\pd_R(\widehat{R})$?
\end{question}

The work of Raynaud and Gruson \cite{GR} includes the finiteness of $\pd_R(\widehat{R})$, when the ring is local. Despite the naturality of the question, however, comparatively little attention has been paid to an explicit computation of $\pd_R(\widehat{R})$ in the literature.
 Matlis \cite[Cor 5.3.1]{mat2} proved that $\pd_R( \overset{\sim} R )\leq \pd_R(Q)$. Suppose $R$ is not complete in $R$-topology. We show
$$\pd_R(Q) = \pd_R( \overset{\sim} R )\leq \pd_R(\widehat{R}).$$
It is easy to construct
situation for which  $\pd_R(\widehat{R})=1$, for example non-complete 1-dimenstional regular local rings do the job.
We present explicit examples for which 
 $\pd_R(\widehat{R})=2$ and also  $\pd_R(\widehat{R})=3.$
We close Section 7 by a calculation of  $\otimes_n\overset{\sim}R$, which may be useful in the construction of Amitsur resolution.
In Section 8, we present some explicit computation of $\Ext_R^\ast(\widehat{R},R)$. This may extend  \cite{wag}.
Let $(R,\fm)$ be a complete local domain of dimension one, and let $S$ be torsion-free and of rank one. It was a conjecture of Matlis that $Q/R\otimes_R S$ is indecomposable. 
In \cite{d} he was successful to show that it is indeed indecomposable. In the final section, we reconsider this and connect it to the current research  topic of indecomposablity
of top local cohomology modules.
In particular, this paper deals with the Matlis'  quadratic $(Q,Q/R, \widehat{R},\overset{\sim}R).$
There are some related results that not discussed.
We invite the interested reader to see \cite{moh} for more related results in the style of Matlis, specially the homologies of $Q$.

\section{Weakly co-torsion}
All rings in the sequel are commutative. We consider both noetherian and non-noetherian rings. The term quasilocal denotes a ring with a unique maximal ideal. By local ring we mean a noetherian quasilocal ring. We start by recalling:

\begin{definition}\label{1} (Matlis)
Let $(R,\frak{m})$ be an integral domain with field of fraction $Q$, and $M$ be a nonzero $R$-module.
\begin{itemize}
\item[(i)]  $M$ is $h$-reduced if $\Ext_R^0 (Q,M)=0$.
\item[(ii)]   $M$ is called weakly cotorsion if $\Ext_R^{1}(Q,M)=0$.
\item[(iii)]   $M$ is called cotorsion if $\Ext_R^{0}(Q,M)=\Ext_R^{1}(Q,M)=0$.
\item[(iv)]   $M$ is called strongly cotorsion if $\Ext_R^{i}(Q,M)=0$ for all $i>0$.
\end{itemize}
\end{definition}

By $\pd_R(-)$ (resp. $\id_R(-)$)  we mean projective dimension (resp. injective dimension).
 
\begin{remark}\label{14}
	Let $(R,\frak{m})$ be an integral domain. Matlis proved that if $\dim(R)=1$, then $\pd_R (Q)=1$. Also, he proved the following  are equivalent:
	\begin{itemize}
		\item[(i)] $\Ext_R^1 (Q/R, B)$ is weakly cotorsion and $h$-reduced.
		\item[(ii)] $\pd_R (Q)=1$.
		\item[(iii)] $\pd_R (Q/R)=1$.
	\end{itemize}

\end{remark}
\begin{notation}
	Let $E_R(-)$ be the injective envelope of $(-)$.
\end{notation}
 \begin{notation}Let $(R,\fm,k)$ be quasi-local. The notation $(-)^{ \vee}$ stands for the Matlis functor defined by $\Hom_R(-,E_R(k))$.\end{notation}

 \begin{observation}\label{6}  Let $(R,\frak{m})$ be a complete local integral domain and $M$ and $N$ are finitely generated $R$-modules. Then $\Ext_R^{1}(M,N)$ is weakly cotorsion.
 \end{observation}
 \begin{proof}
 	Let $E:= \Ext_R^{1} (M,N) $ and recall that $E$ is finitely generated $R$-module. By Matlis duality $E=E^ {\vee \vee}$, then
 	
 	$$\Ext_R^{1}(Q,E)= \Ext_R^{1} (Q, (E^\vee )^\vee) \cong \Tor_1^{R} (Q, E^ \vee )^\vee.$$
 	As $Q$ is flat $R$-module, $\Tor_1^{R} (Q,-)=0$ and so $E$ is weakly cotorsion.
 \end{proof}

\begin{remark}
The  previous assumptions can be substantially weakened to
$$
\boxed{\begin{gathered}
		R\text{ any commutative domain},\quad M\text{ finitely generated},
		\quad N\text{ arbitrary}\\
		\Longrightarrow\quad
		\Ext_R^p\bigl(Q,\Ext_R^q(M,N)\bigr)=0
		\quad(p\geq0,\ q\geq1).
\end{gathered}}
$$
Indeed, choose a free submodule $F\subseteq M$ spanning $Q\otimes_R M$.
The finitely generated torsion module $M/F$ is annihilated by some
$0\ne r\in R$. Thus $r\,\mathrm{id}_M$ factors as $M\to F\hookrightarrow M$,
which annihilates all positive $\Ext_R^q(M,N)$. Since $r$ is invertible on
$Q$, all their Ext groups against $Q$, including Hom, vanish.
This proof requires neither finite presentation nor Noetherianity.

If $\pd_R Q\leq1$, then for arbitrary $M,N$,
$$
\Ext_R^p\bigl(Q,\Ext_R^q(M,N)\bigr)=0
\qquad(p\geq1,\ q\geq1).
$$
The spectral sequence
$$
\Ext_R^p\bigl(Q,\Ext_R^q(M,N)\bigr)
\Longrightarrow \Ext_R^{p+q}(Q\otimes_R M,N)
$$
has only columns $p=0,1$. Also $Q\otimes_R M$ is a direct sum of copies
of $Q$ and hence has projective dimension at most one. The term
$\Ext_R^1(Q,\Ext_R^q(M,N))$ therefore injects into a zero abutment group
for $q\geq1$.
This subsumes Proposition 2.9 over countable domains, removes both its
torsion hypothesis on $M$ and its injective-dimension hypothesis on $N$,
and includes all one-dimensional Noetherian local domains.
The accompanying note gives the double-complex proof and versions for
general localization.	
\end{remark}

\begin{observation}\label{3}  Let $(R,\frak{m})$ be a 1-dimensional  local integral domain and $\id (N) < \infty $ and $N$ be of finite Bass number. Let $M$ be an  $R$-module. Then $\Ext_R^{1}(M,N)$ is weakly cotorsion.
\end{observation}
\begin{proof}
 We set
$$\eta : 0\longrightarrow N \overset{\varphi}{\longrightarrow} E(N) \longrightarrow \Omega^{-1} N \longrightarrow 0.$$
Without loss of generality, we assume $N$ is not injective. Since $N$ is not injective then $\id(N)=1$.
We know $\Omega^{-1} N$ is injective. By Matlis' decomposition theory 
 $\Omega^{-1} N:= \underset{\frak{p} \in \Spec(R)}\oplus E(R/\frak{p})^ {\mu(\frak{p},N)}. $
Recall that $\Spec(R)=\lbrace 0,\frak{m}\rbrace$, 
so $$ \Omega^{-1} N=E(R)^{{m}_1} \oplus E(R/\frak{m})^{{m}_2}=Q^{{m}_1} \oplus E(R/\frak{m})^{{m}_2} $$
where $m_i < \infty $. Now apply $\Hom(M,-)$ to the sequence $\eta $, we have
$$\Hom_{R}(M,\Omega^{-1} N) \overset{\psi}{\longrightarrow} \Ext_{R}^1 (M,N) \longrightarrow \Ext_{R}^1(M,E(N)).$$
Let $K:=\ker (\psi)$, then
$$0 \longrightarrow K \longrightarrow \Hom_{R}(M,\Omega^{-1} N) \longrightarrow \Ext_{R}^1 (M,N) \longrightarrow 0.$$
Applying $\Hom_{R}(Q,-)$ to the above sequence, we have
$$\Delta : \Ext_R^{1} (Q,\Hom_{R}(M,\Omega^{-1} N)) \longrightarrow \Ext_R^{1} (Q, \Ext_{R}^1 (M,N)) \longrightarrow \Ext_R^{2} (Q,K).$$
In view of \cite[Theorem 4.2]{E}  $\pd(Q) = 1$.   This shows that $\Ext_R^{2} (Q,K)=0$.
Now it is enough to show that $\Ext_R^{1} (Q,\Hom_{R}(M,\Omega^{-1} N))=0$.
First we analyze that $\Hom_{R}(M,\Omega^{-1} N)$

$$\begin{array}{ll}
\Hom_{R}(M,\Omega^{-1} N)&=\Hom_R(M,Q)^{\oplus {m}_{1}} \oplus \Hom_R(M,E(R/\frak{m}))^{\oplus {m}_{2}}\\
&= (\oplus Q) \oplus  ( \oplus \Hom_R (M,E(R/\frak{m}))  =(\oplus Q) \oplus ( \oplus M^\vee).
\end{array}$$

Now
$$\Ext_R^{1} (Q,\Hom_{R}(M,\Omega^{-1} N)) = \oplus \Ext _R ^{1}(Q,Q) \oplus (\oplus \Ext_R^{1} (Q, M^\vee).$$
Recall that  $Q$ is injective. This implies that $\Ext _R ^{1}(Q,Q)=0$.
Since $\Ext_R^{1} (Q, M^\vee) \cong \Tor_1 ^{R} (Q, M^\vee) $, and due to the flatness of $Q$, it is zero.
\end{proof}

Let us recover one of delicate result of Matlis \cite[10.4]{mat2} via an elementary  calculation (see {Theorem} \ref{matlis} for a wast generalization):
\begin{proposition}\label{9} Let $(R,\frak{m})$ be a countable integral domain. Let $M$ be torsion $R$-module and $\id (N) < \infty $. Then $\Ext_R^{i}(M,N)$ is weakly cotorsion for $i=0,1$. In fact,  $\Ext_R^{>0} (Q,\Ext_R^{i}(M,N))=0.$ 
\end{proposition}
\begin{proof}
The proof is by induction on $n:=\id(N) < \infty $. The case $n=0$ is clear: In view of \cite[A.11]{matcanada} we have
$$\Ext_R^1 (Q, \Hom_R (M, \emph{injective})) = \Hom_R (\Tor_1^R (Q,M), \emph{injective})\quad(\ast)$$Since $Q$ is flat, the corresponding $\Tor$ is zero. By $(\ast)$, $\Ext_R^1 (Q, \Hom_R (M,N))=0$. Also, $\Ext_R^1 (Q, \Ext^1_R (M,N))=0, $ as $N$ is injective.
We may and do assume $n>0$.
Now, let
$$\eta: 0\longrightarrow N \overset{\varphi}{\longrightarrow} E^0 \longrightarrow \Coker (\varphi) \longrightarrow 0,$$
where $E^0 := E_R(N)$ and let $E_1:=\Coker (\varphi) $.
Applying $\Hom (M,-)$ to $\eta$, we have
$$\ldots \longrightarrow \Hom_R(M,E_1) \overset{\psi}{\longrightarrow} \Ext_R^1 (M,N) \longrightarrow \Ext_R^1(M,E^0).$$
Since $E^0$ is injective, we have $\Ext_R^1(M,E^0)=0$. Let $K:=\Ker (\psi)$. We have the following exact sequence
$$\Psi: 0\longrightarrow K \longrightarrow \Hom_R (M,E_1) \longrightarrow \Ext_R^1 (M,N) \longrightarrow 0 .$$
Applying $\Hom_R (Q,-)$ to $\Psi $ yields that
$$\Ext_R^1 (Q, \Hom_R (M, E_1)) \longrightarrow \Ext_R^1 (Q, \Ext_R^1 (M, N)) \longrightarrow \Ext_R^2 (Q,K) .$$
Since $R$ is countable, the flat module $Q$ is as well. This implies that $\pd_R (Q)=1 $ and consequently,  $\Ext_R^2 (Q,K)=0$.
Since $\id_R (E_1) = \id_R (N) -1$, and by applying the inductive hypothesis, we may assume that $\Ext_R^1 (Q, \Hom(M, E_1))=0$.
So, $$\Ext_R^1 (Q, \Ext_R^1 (M, N))=0\quad(+)$$
Now, we show that $\Ext_R^0 (M,N)$ is cotorsion.
Indeed, we show that the following sequence
$$0 \longrightarrow \Hom_R (M,E^0) \longrightarrow \Hom_R (M,E_1) \longrightarrow K \longrightarrow 0$$
is exact.
Applying $\Hom_R (M,-) $ to $\eta$
$$0 \longrightarrow \Hom_R (M,N) \longrightarrow \Hom_R (M,E^0) \longrightarrow \Hom_R (M,E_1) \longrightarrow \Ext_R^1 (M,N) \longrightarrow \Ext_R^1 (M,E^0).$$
We break down it into two short exact sequences
	\begin{enumerate}
	\item[a)]  $  0 \longrightarrow \Hom_R (M,N) \longrightarrow \Hom_R(M,E^0) \longrightarrow L \longrightarrow 0$ and
	\item[b)]  $ 0 \longrightarrow L \longrightarrow \Hom_R(M,E_1) \longrightarrow \Ext_R^1 (M,N)$.
\end{enumerate}

Applying $\Hom_R (Q,-)$ to the exact sequence $b)$ yields that
$$0 \longrightarrow \Hom_R(Q,L) \longrightarrow \Hom_R(Q ,\Hom_R (M,E_1)).$$
Recall that $M$ is torsion. This gives  $Q \otimes M=0$. So, $$\Hom_R(Q ,\Hom_R (M,E_1)) \cong \Hom_R (Q \otimes M ,E_1)=0.$$  By plugging this in the previous sequence, $\Hom_R (Q,L)=0$.
Applying $\Hom_R (Q,-)$ to  the exact sequence displaced in  $a)$,  yields that
$$0=\Hom_R(Q,L) \longrightarrow \Ext_R^1 (Q,\Hom_R( M,N ) )\longrightarrow \Ext_R^1 (Q, \Hom_R (M, E^0))\quad(\ast\ast)$$
Since $$\Ext_R^1 (Q, \Hom_R (M, E^0)) \cong \Tor_1^{R} (Q, M)^{\vee_{E^0}}$$ and that $Q$ is flat, we observe that $\Tor_1^{R} (Q, M)$ is zero.
In view of $(\ast\ast)$ we conclude that $$ \Ext_R^1 (Q,\Hom_R( M,N )) = 0\quad(+,+)$$
By $(+)+(++)$ we get the desired claim. The particular case is trivial. Indeed,
since $\pd_R(Q)=1$ then $\Ext_R^ {\geq 2} (Q,-)=0$. The proof is now complete.
\end{proof}

For the reminder of this section $R$ is any commutative ring, $S\subseteq R$ is a
multiplicative subset with $0\notin S$, and $A=S^{-1}R$. We say that $C$
is $S$-weakly cotorsion if $\Ext_R^1(A,C)=0$. This is the standard relative
terminology used, for example, in \cite[Section 1]{BP}.

\begin{lemma}\label{annihilator}
	If $sC=0$ for some $s\in S$, then $\Ext_R^p(A,C)=0$ for every $p\geq0$.
\end{lemma}
\begin{proof}
	The central action of $s$ on $\Ext_R^p(A,C)$ can be computed from either
	argument. In the second argument it is zero. In the first argument it is
	an automorphism, since multiplication by $s$ on $A$ is an automorphism.
	Thus the group is zero. This also applies to $p=0$.
\end{proof}

\begin{theorem}\label{factor}
	Suppose that, for some $s\in S$, multiplication by $s$ on $M$ factors
	through a projective $R$-module:
$
	M\xrightarrow{u}P\xrightarrow{v}M,$
 with $vu=s\,\mathrm{id}_M.$
	Then for every $R$-module $N$ and every $q\geq1$ one has
	$s\Ext_R^q(M,N)=0$. In particular
$
	\Ext_R^p\bigl(A,\Ext_R^q(M,N)\bigr)=0$
 where $p\geq0,\ q\geq1$.
\end{theorem}
\begin{proof}
	Contravariance shows that multiplication by $s$ on $\Ext_R^q(M,N)$
	factors through the zero module $\Ext_R^q(P,N)$. Now apply
	Lemma \ref{annihilator}.
\end{proof}

\begin{theorem}\label{finite}
	Let $R$ be a commutative domain, with no Noetherian or local assumption.
	If $M$ is finitely generated, there is $0\ne r\in R$, depending only on $M$,
	such that $r\Ext_R^q(M,N)=0$ for every $N$ and every $q\geq1$.
	Consequently all these extension modules are $h$-reduced and strongly
	$Q$-orthogonal.
\end{theorem}
\begin{proof}
	Choose elements $m_1,\ldots,m_t\in M$ whose images form a $Q$-basis of
	$Q\otimes_R M$. Their span $F=\sum_{j=1}^t Rm_j$ is free: an $R$-linear
	relation becomes a relation between the chosen $Q$-basis elements.
	The module $M/F$ is finitely generated and torsion. A product of
	annihilators of finitely many generators gives $0\ne r\in R$ with
	$r(M/F)=0$. Hence $rM\subseteq F$ and
$
	M\xrightarrow{m\mapsto rm}F\hookrightarrow M
$
	has composite $r\,\mathrm{id}_M$. Theorem \ref{factor}, with
	$S=R\setminus\{0\}$, proves the assertion. If $Q\otimes_R M=0$, use
	$F=0$ in the same argument.
\end{proof}

\begin{corollary}\label{sums}
	Over any commutative domain, the conclusion
 $\Ext_R^p\bigl(Q,\Ext_R^q(M,N)\bigr)=0	$ for
 $p\geq0,\ q\geq1$
also holds when $M$ is a direct summand of an arbitrary direct sum of
	finitely generated modules. No uniform annihilator for the summands is
	required.
\end{corollary}
\begin{proof}
	For $M=\bigoplus_i M_i$, one has
	$\Ext_R^q(M,N)\cong\prod_i\Ext_R^q(M_i,N)$.
	For every $p\geq0$, $\Ext_R^p(Q,-)$ commutes with products: apply
	$\Hom_R(P_\bullet,-)$ to a projective resolution of $Q$ and use exactness
	of products in the module category. Apply Theorem \ref{finite} component wise,
	then pass to direct summands.
\end{proof}

\begin{proposition}\label{fp}
	Let $R,S,A$ be as above. Suppose $M$ is finitely presented and
	$A\otimes_R M$ is projective over $A$. Then $M$ satisfies the factorization
	criterion of Theorem \ref{factor}. Thus all positive extension modules
	are $S$-weakly cotorsion, and in fact orthogonal to $A$ in all degrees,
	including degree zero.
\end{proposition}
\begin{proof}
	The finitely generated projective $A$-module $M_S$ is a direct summand of
	$A^t$ for some finite $t$. Write the splitting maps as
	$a:M_S\to A^t$, $b:A^t\to M_S$, with $ba=\mathrm{id}_{M_S}$.
	For a finitely presented source, localization commutes with $\Hom$. Clearing denominators yields $s_1,s_2\in S$
	and maps $u:M\to R^t$, $v:R^t\to M$ with
	$u_S=s_1a$ and $v_S=s_2b$. Consequently
	$(vu-s_1s_2\mathrm{id}_M)_S=0$.
	Since $M$ is finitely generated, finitely many annihilators of the images
	of its generators give $s_3\in S$ such that
	$s_3vu=s_1s_2s_3\mathrm{id}_M$. Use $u$, $s_3v$, and
	$s=s_1s_2s_3$ in Theorem \ref{factor}.
\end{proof}

 {Arbitrary modules and projective dimension one}
We first prove the relevant spectral sequence, so that the hypotheses
and convergence used in the argument are explicit.

\begin{lemma}\label{spectral}
	Let $R$ be commutative and $F$ a flat $R$-module. There is a first-quadrant
	spectral sequence
\[
	E_2^{p,q}=\Ext_R^p\bigl(F,\Ext_R^q(M,N)\bigr)
	\ \Longrightarrow\ \Ext_R^{p+q}(F\otimes_R M,N).\tag{*}\label{ss}
	\]
	If $\pd_R F\leq1$, then for $n\geq1$ it gives natural short exact sequences
$$(+):\quad
		0\to\Ext_R^1\bigl(F,\Ext_R^{n-1}(M,N)\bigr)
		\to\Ext_R^n(F\otimes_R M,N) 
		\to\Hom_R\bigl(F,\Ext_R^n(M,N)\bigr)
		\to0.
$$
\end{lemma}
\begin{proof}
	Take projective resolutions $P_\bullet\to F$ and $L_\bullet\to M$ and
	consider the first-quadrant cochain double complex
	$$
	D^{p,q}=\Hom_R(P_p,\Hom_R(L_q,N))
	\cong\Hom_R(P_p\otimes_R L_q,N).
	$$
	First taking cohomology in $q$, then in $p$, gives the displayed $E_2$-page.
	The total complex $\operatorname{Tot}(P_\bullet\otimes_R L_\bullet)$ is a
	projective resolution of $F\otimes_R M$: tensor products of projective
	are projective over a commutative ring, and flatness of $F$ removes the
	positive Tor homology. Thus the abutment is as claimed. Each total degree
	has a finite filtration, giving convergence.
	
	If $\pd_R F\leq1$, choose $P_\bullet$ concentrated in degrees $0,1$.
	Only columns $p=0,1$ survive. Every differential $d_r$, $r\geq2$, is zero,
	and the two filtration quotients in total degree $n$ give (+).
\end{proof}

\begin{theorem}\label{matlis}
	Let $R$ be any commutative domain with $\pd_R Q\leq1$.
	For arbitrary $R$-modules $M,N$ and integers $q,p\geq1$,
$
	\Ext_R^p\bigl(Q,\Ext_R^q(M,N)\bigr)=0.
$
	In particular Question 1.1 has an affirmative answer over every Matlis
	domain, without any finiteness or torsion conditions on $M,N$.
\end{theorem}
\begin{proof}
	The $Q$-vector space $Q\otimes_R M$ is a direct sum of copies of $Q$.
	It therefore has projective dimension at most one over $R$.
	Apply (+) from {lemma} \ref{spectral} with $F=Q$ and $n=q+1\geq2$. Its middle term is zero,
	so $\Ext_R^1(Q,\Ext_R^q(M,N))=0$. The cases $p\geq2$ follow directly
	from $\pd_R Q\leq1$.
\end{proof}

\begin{corollary}\label{instances}
	Theorem \ref{matlis} applies to every countable commutative domain and to
	every one-dimensional Noetherian local domain. 
\end{corollary}
\begin{proof}
	A localization at a countable multiplicative set has a telescope resolution
	of length one. Indeed, writing the localization as a sequential direct limit
	$R\xrightarrow{s_1}R\xrightarrow{s_2}R\to\cdots$ gives
	$$
	0\longrightarrow\bigoplus_{n\geq0}R
	\xrightarrow{\delta}\bigoplus_{n\geq0}R
	\longrightarrow S^{-1}R\longrightarrow0,
	\qquad \delta(e_n)=e_n-s_{n+1}e_{n+1}.
	$$
	Enumerating all nonzero elements gives the countable-domain case.
	For a one-dimensional Noetherian local domain, any $0\ne x\in\m$
	satisfies $R_x=Q$: every nonzero prime is maximal and contains $x$.
	Apply the same resolution to the powers of $x$.
\end{proof}

\begin{theorem}\label{semisimple}
	Let $R$ be any commutative ring and $A=S^{-1}R$. Suppose that $A$ is
	semisimple and $\pd_R A\leq1$. Then for arbitrary $M,N$,
 $	\Ext_R^p\bigl(A,\Ext_R^q(M,N)\bigr)=0$ 
for $p\geq1,\ q\geq1$.
	More generally, the same conclusion holds for a given $M$ whenever
	$A\otimes_R M$ is projective over $A$, with no finite generation assumption.
\end{theorem}
\begin{proof}
	A projective $A$-module is a direct summand of a direct sum of copies of
	$A$, so it has projective dimension at most one over $R$. Apply
	Lemma \ref{spectral}, just as in Theorem \ref{matlis}. If $A$ is semisimple,
	every $A$-module is projective.
\end{proof}

\begin{example}[A ring with zero divisors]
	Take $R=(k[x,y]/(xy))_{(x,y)}$ and $s=x+y$.
	The ring $R_s$ is reduced, Noetherian, and zero-dimensional, hence a
	finite product of fields (in fact $k(x)\times k(y)$). It is semisimple,
	and the telescope resolution gives $\pd_R R_s\leq1$.
	Theorem \ref{semisimple} applies to arbitrary $M,N$ over this non-domain.
\end{example}

\begin{remark}
	With $F$ flat and $\pd_R F\leq1$, 	Lemma \ref{spectral} implies
	$$
	\Ext_R^1(F,\Ext_R^q(M,N))=0
	\quad\text{if}\quad
	\Ext_R^{q+1}(F\otimes_R M,N)=0.
	$$
	For instance this holds for $q\geq d$ if $\pd_R(F\otimes_R M)\leq d$,
	and for $q=t\geq1$ if $\id_R N\leq t$.
\end{remark}

The following identity gives a correct replacement for the dimension
shifting used in that proposition.

\begin{proposition}\label{obstruction}
	Let $R$ be commutative, $F$ flat, and $\id_R N\leq1$.
	For every $p\geq1$ there is an isomorphism
	 $
	\Ext_R^p\bigl(F,\Ext_R^1(M,N)\bigr)
	\cong\Ext_R^{p+2}\bigl(F,\Hom_R(M,N)\bigr).
	 $
	In particular,  $\pd_R F\leq2$, then $\Ext_R^1(M,N)$ is orthogonal
	to $F$ in every positive degree.
\end{proposition}
\begin{proof}
	Choose $0\to N\to I^0\to I^1\to0$ with $I^0,I^1$ injective.
	Put $H=\Hom_R(M,N)$, $C^j=\Hom_R(M,I^j)$,
	$E=\Ext_R^1(M,N)$, and $B=\im(C^0\to C^1)$.
	The correct two short exact sequences are
	$$
	0\longrightarrow H\longrightarrow C^0\longrightarrow B\longrightarrow0,
	\qquad
	0\longrightarrow B\longrightarrow C^1\longrightarrow E\longrightarrow0.
	$$
	For $i>0$, tensor--Hom adjunction and injectivity give
	$\Ext_R^i(F,C^j)
	\cong\Hom_R(\Tor_i^R(F,M),I^j)=0.$
	The two long exact sequences now yield
	$\Ext_R^p(F,E)\cong\Ext_R^{p+1}(F,B)
	\cong\Ext_R^{p+2}(F,H)$.
\end{proof}

\begin{example}\label{negative-localization}
	Let $R=k[x,y]$, $A=R_x$, $M=R/yR$, and $N=R$. Then $\pd_R A\leq1$
	and $M$ is finitely presented, but
	$ \Ext_R^1(M,R)\cong k[x]$ and $\Ext_R^1(A,k[x])\cong k[[x]]/k[x]\ne0. $
	Thus the universal $S$-relative assertion is false, even over a regular
	Noetherian domain with finitely presented $M$.
\end{example}
\begin{proof}
	The first equality follows from the resolution
	$0\to R\xrightarrow{y}R\to M\to0$.
	Tensoring the telescope resolution of $R_x$ with $R/yR$ gives the
	telescope resolution of $k[x,x^{-1}]$ over $k[x]$. Hence the second Ext
	is $\Ext_{k[x]}^1(k[x,x^{-1}],k[x])$.
	Writing $B=k[x]$, the telescope computes this as the cokernel of
	$$
	\prod_{n\geq0}B\longrightarrow\prod_{n\geq0}B,
	\qquad (b_n)\longmapsto(b_n-xb_{n+1}).
	$$
	The map $(a_n)\mapsto\sum_{n\geq0}x^na_n$ modulo $B$ identifies that
	cokernel with $k[[x]]/B$: it is surjective, it takes the displayed image
	into $B$, and if the sum is a polynomial $b_0$, then
$
	b_n=\frac{b_0-\sum_{i<n}x^ia_i}{x^n}\in B
$
	solves $b_n-xb_{n+1}=a_n$.
\end{proof}

\section{An application: the splitting result}

We need the following three standard lemmas:
\begin{lemma}\label{10}
Suppose $R$ is an integral domain and $I$ is an injective and $M$ any $R$-module. Then $\Hom_R (I,M)$ is torsion-free.
\end{lemma}
\begin{proof}
Let $r \in R$ be such that $\exists f: I \longrightarrow M $ and that  $r.f = 0$. Suppose $f$ is not zero, then there exists $x \in I$ such that $f(x) \neq 0$. We know injective modules over an integral domain are divisible, so there exists $y \in I$ such that $x = ry$.
Now  $f(x) = f(ry) = rf(y) = (rf)(y) =0,$ as claimed.
\end{proof}

\begin{lemma}\label{11}
Suppose $(R,\frak{m}) $ is  local, $1$-dimensional and $M$ is a finitely generated weakly cotorsion $R$-module. Suppose in addition that
	\begin{enumerate}
	\item[a)]
$\id(M)<\infty$    and 	\item[b)] $N$ is any torsion-free module. \end{enumerate}Then $\Ext_R^1 (N,M)=0$.
\end{lemma}
\begin{proof}
Since $N$ is torsion-free then $\frak{m} \notin \Ass(N) $. Let
 $N \hookrightarrow E(N)= \oplus_{\frak{p} \in \Ass(N)} E(R/\frak{p})^ {\mu (\frak{p},N)}.$ 
Since $\frak{m} \notin \Ass(N)$ and $\Spec (R)= \lbrace 0,\frak{m}\rbrace $ then $\Ass(N)= \lbrace0\rbrace $. Now
 $N \hookrightarrow E(N) = \oplus Q ^t, $
where $t:=\mu ( 0 ,N)$.
Let     $0 \longrightarrow N \overset{\rho}{\longrightarrow} E(N) \longrightarrow C:=\Coker (\rho) \longrightarrow 0$. Applying $\Hom(-,M)$ to it, yields the following exact sequence
$$\Ext_R^1 (E(N),M) \longrightarrow \Ext_R^1 (N,M) \longrightarrow \Ext_R^2 (C,M)\quad(\ast)$$
Since $M$ is finite and of finite injective dimension, we deduce that $\id(M)\leq\depth(R)=1$. By this,
$\Ext_R^2 (C,M)=0.$ Recall that $M$ is weakly co-torsion, so $$\Ext_R^1 (E(N),M)=\Ext_R^1 (\oplus_t Q,M)=\prod_t\Ext_R^1 (Q,M)=0.$$
Putting these in $(\ast)$, we conclude that $\Ext_R^1 (N,M)=0$.
\end{proof}

\begin{lemma}\label{21}
	Let $(R,\frak{m})$ be a local integral domain with fraction field $Q$. Then $Q/R$ is torsion.
\end{lemma}
\begin{proof}
	Let $\alpha + R \in Q/R$ be non-zero. then $\alpha = r/s $ where $r \in R$ and $s \in R \setminus \lbrace0\rbrace $.
	Since $\alpha + R \neq 0$, then $\alpha \notin R$ and $s \in \frak{m}$. (any element of $R \setminus \lbrace \frak{m} \rbrace$ is invertible). Now
	$$\begin{array}{ll}
	s.( \alpha + R )&=(s \alpha) +R = ( s . r/s) + R =r + R = 0.
	\end{array}$$
By definition, $Q/R$ is torsion.
\end{proof}
\begin{lemma}\label{2of3}
	Let $0\to M\to L\to N\to 0$ be an exact sequence of finitely generated  modules where $N$ and $M$ are  as Lemma \ref{11}. Then $L=N\oplus M$.
\end{lemma}

\begin{proof}
Apply $\Hom_R(N,-)$ to  $0\to M\to L\to N\to 0$ gives us  $\Hom_R(N,L)\stackrel{\pi}\to \Hom_R(N,N)\to\Ext^1(N,M).$ In the light of Lemma \ref{11} we see $\Ext^1(N,M)=0$, i.e., $\pi$ is surjective. There is $f:N\to L$ such that $\pi(f)=\id_{N}$. So, $L=N\oplus M$.
\end{proof}

\begin{notation} Let $M$ be an $R$-module. 
	\begin{enumerate}
	\item[i)]   By	$t(M) $ we mean $\langle m \in M | rm= 0 \emph{ for some nonzero r in R}\rangle$.
	\item[ii)]  By $H_{\frak{m}}^{i}(M)$ we mean the $i^{th}$ local cohomology module of $M$ with respect to $\fm$.
\end{enumerate}

\end{notation}
\begin{proposition}\label{12}
Let $(R,\frak{m})$ be a 1-dimensional Gorenstein local  integral domain and $M$ be $R$-module such that $\id_R(t(M)) < \infty $.
$$\Ext^1_R(Q/R,M)=\Ext^1_R(Q/R , t(M)) \oplus \Ext^1_R(Q/R ,M/ t(M)).$$
\end{proposition}
\begin{proof}Let $\overline{M} := M/t(M)$.
We apply $\Hom_R(Q/R,-)$ to the short exact sequence
 $0 \to t(M) \to M \to \overline{M} \to 0,$ 
and deduce the following exact sequence
$$\Hom_R (Q/R, \overline{M}) \to \Ext_R^1 (Q/R, t(M)) \to \Ext_R^1(Q/R,M) \to \Ext_R^1 (Q/R ,\overline{M} ) \to \Ext_R^2 (Q/R,M).$$
Since $R$ is Gorenstein, $\id(R)=1$. Also, $\pd(Q)=1$. In view of $0\to R\to Q \to Q/R \to 0$ we see $Q/R$ is injective and of projective dimension at most one. In particular,  $\Ext_R^2 (Q/R,M)=0$.
Thanks to Lemma \ref{21}, $Q/R$ is torsion. Since $\overline{M}$ is torsion-free,
we know $\Hom_R (Q/R , \overline{M})=0$. Thus, the sequence
$$0 \longrightarrow \Ext_R^1 (Q/R, t(M)) \longrightarrow \Ext_R^1(Q/R,M) \longrightarrow \Ext_R^1 (Q/R ,\overline{M} ) \longrightarrow 0,$$is exact.
Now we claim that $\Ext_R^1(Q/R , \overline{M})$ is torsion-free.
Indeed let $E:= E(\overline{M}) $ be the injective envelope of $\overline{M}$.
Recall that $E=\oplus Q$. This gives $\Hom_R (Q/R, E ) =0$, as the first component is torsion and the second is torsion-free. Let $0\to \overline{M}\to E\to E/\overline{M}\to 0$ be the natural sequence. By the standard shifting:
$$0=\Hom_R (Q/R, E )\to  \Hom_R (Q/R, E/\overline{M} )\stackrel{}\lo\Ext_R^1 (Q/R , \overline{M} )\lo\Ext_R^1 (Q/R , E )=0, $$and apply Lemma \ref{10} to see $\Ext_R^1(Q/R , \overline{M})$ is torsion-free.
Thanks to \cite[Theorem 10.1]{mat2} we know
$\Ext_R^1 (Q/R, t(M))$
is weakly co-torsion.

Since the ring is Gorenstein and of dimension one,  $Q/R \cong E(k)$. Recall that $\pd_R(Q)=1$, and in view of the exact sequence $ 0 \to R \to Q \to Q/R \to 0$ we observe that  $\pd_R(E(k))=\pd_R(Q/R)=1$. This shows $\Ext_R^1 (E(k), -)$ is right exact. By a theorem of Jensen (see  \cite{moh}) we know that for finitely generated $R$-module $N$
 $N \cong \hat{N} \cong \Ext_R^1 (H_{\fm}^1 (R) ,N) \cong \Ext_R^1 (Q/R,N).$ 
Now apply this for $N:= t(M) $ and $ \overline{M}$.
In particular, we are in the situation of 
Lemma \ref{2of3}. This implies 
that
$$\Ext^1_R(Q/R,M)=\Ext^1_R(Q/R , t(M)) \oplus \Ext^1_R(Q/R , M/t(M)),$$
as claimed.
\end{proof}

\begin{corollary}\label{16}
Let $(R,\frak{m})$ be a 1-dimensional Gorenstein complete local integral domain and $M$ and $t(M)$ be a finitely generated $R$-module of finite injective dimension. Then $M \cong t(M) \oplus \overline{M}$.
\end{corollary}
\begin{proof}
Recall that 
 $N \cong \hat{N} \cong \Ext_R^1 (H_{\fm}^1 (R) ,N) \cong \Ext_R^1 (Q/R,N) $ for finitely generated $R$-module $N$.
Now apply this for the triple $N:=M , t(M) $ and $ \overline{M}$. By Proposition \ref{12}, we get
$M \cong t(M) \oplus \overline{M}$.
\end{proof}

The  assumption $t(M)$ is of finite injective dimension is needed:
\begin{example}
Let
$R:=k[[x,y]]/(y^2-x^3)\cong k[[t^2,t^3]]$ and $
M:=R^2/R(-y,x).
$
This is a complete one-dimensional Gorenstein domain. The sequence
$$
0\longrightarrow R\xrightarrow{\,a\mapsto(-ay,ax)\,}R^2
\longrightarrow M\longrightarrow0
$$
is exact. Thus $M$ is finitely generated and has projective dimension
at most one. It also has finite injective dimension: since $\id_R R=1$,
the long exact Ext sequence gives $\Ext_R^i(X,M)=0$ for $i>1$ and all $X$.
Consider the surjection
$$
\pi:M\longrightarrow\m=(x,y),\qquad
[(u,v)]\longmapsto xu+yv.
$$
It is well-defined because $x(-y)+yx=0$. Both modules have rank one,
so $\ker\pi$ is torsion. Conversely $\m$ is torsion-free, so every torsion
element of $M$ lies in $\ker\pi$. Hence
$$
t(M)=\ker\pi,\qquad M/t(M)\cong\m.
$$
This torsion submodule is nonzero: $[(-x^2,y)]$ is killed by $x$, since
$x(-x^2,y)=y(-y,x)$, and is not zero in $M$. If it were zero, its second
coordinate would imply $y=ax$ for some $a\in R$, whereas $y/x=t\notin R$.

If Corollary \ref{16} were true without without any assumption on $t(M)$, $\m$ would be a direct summand of $M$ and
would have finite projective dimension. Since $\depth_R\m=\depth R=1$,
Auslander--Buchsbaum would give $\pd_R\m=0$. A finitely generated
projective module over a local ring is free; because $\m$ has rank one,
it would be principal. But $\dim_k\m/\m^2=2$, a contradiction.
Therefore
$
M\not\cong t(M)\oplus M/t(M).
$
\end{example}
\section{Homology of Q}
Here is the higher dimensional  version of \cite[Proposition 5]{matcanada}:
\begin{proposition}\label{mtq}
Let $(R, \fm )$ be a complete local integral domain of dimension $d$.  Then $Q^{\vee} = \Hom_{R}( Q , E(k) )$ is $Q$-vector space. Also $Q^{\vee} = \oplus_t Q$ where
$$
t	=\left\{\begin{array}{ll}
1 &\mbox{if }  d=1 \\
\infty	&\mbox{if }  d>1.
\end{array}\right.
$$
\end{proposition}

\begin{proof}
By Matlis, $d=1$ if and only if $Q^{\vee} = Q$. Now assume $ d > 1$ and we are going to show $Q^{\vee} = Q^{t}$ where $ t = \infty $. Suppose on the way of contradiction $ t < \infty$.
By \cite[31.2]{mat}, there is a countable family $\Sigma$ of height one prime ideals and
by countable prime avoidance $\frak{m} \nsubseteq \underset{\fp \in \Sigma}\cup \fp$.
Let $S:= \frak{m} \setminus \underset{\fp \in \Sigma}\cup \fp $ which is not empty.
Let $x,y \in S $ then $xy \in \frak{m}$, and so $xy \in S $. In other words,
$S$ is multiplicative closed.
Let $A:= S^{-1} R$ and let $C = A^{\vee}$. By using adjoint between the functors $\Hom$ and $\otimes$, we see $C$ is injective as an $A$-module.
$$\begin{array}{ll}
\Hom_{A}(C,C)&=\Hom_{A}(C,\Hom_{R}(A,E))= \Hom_{R}(C \otimes A, E)\\
&=\Hom_{R}(C,E)
=C^{\vee}=A^{\vee \vee}.
\end{array}$$
Let $\fp \in \Sigma $ and recall that
$$\begin{array}{ll}
\Hom_{A}(A/\fp A,C)&=\Hom_{A}(A/\fp A,\Hom_{R}(A,E))
= \Hom_{R}(A/\fp A  \otimes A, E)\\
&=\Hom_{R}(A/\fp A,E)=(A/\fp A)^{\vee \vee}.
\end{array}$$

For any $ \mathcal{M} \neq 0 $ we know $ \mathcal{M}^{\vee} \neq 0 $.
We apply this to see $(A/\fp A)^{\vee} = \Hom(A/\fp A,C)\neq0$.
Let $f :A/\fp A \longrightarrow C$ be non-zero. Clearly, $\eta := f (1 +\fp A ) \in C$ is not zero.
Then $\fp \eta =0$ and so $\eta = A / (0 : \eta ) = A/\fp A$, because $pA$ is a maximal ideal of $A$. Now, the assignment  $1+(0 : \eta )\mapsto \eta\in C,$   defines another map
$\eta = A/\fp A \hookrightarrow C. $ This gives an embedding $E( A/ \fp A) \hookrightarrow C$.
The sequence
$$ 0 \longrightarrow E(A/\fp A) \longrightarrow C \longrightarrow \Coker \psi \longrightarrow 0$$
splits. $E (A/\fp A)$ is a directed summand of $C$. There exists $D_{ \fp}$ such that $D_{ \fp} \oplus E(A/\fp A) = C$
$$\begin{array}{ll}
\oplus _{\fp\in\Sigma}R_{ \fp} \hookrightarrow \oplus \widehat{R}_{ \fp} = &\oplus \Hom_{A}(E(A/\fp A),E(A/\fp))\\
&\subseteq \oplus \Hom_{R}(E(A/\fp A), E(A/\fp) \otimes \Hom(E(A/ \fp), D_{ \fp}))\\
&=\Hom_{R}(E(A/ \fp A),E(A/ \fp) \oplus D_{ \fp})\\
&=\sqcap \Hom(E(A/ \fp A),C)=( \oplus E(A/ \fp A),C)\\
&\subseteq \Hom(C,C) =A^{\vee \vee} \subseteq Q^{\vee \vee} =\oplus_t \oplus_t Q.
\end{array}$$

We know that $\oplus E(A/\fp A)$ is injective. In particular, it is a direct summand of $C$.
Recall that $E(A/\fp A) \hookrightarrow C$. This gives an embedding  $ \Sigma E(A/\fp A) \hookrightarrow C.$ This in turns imply that  $ \oplus E(A/\fp A) \hookrightarrow C$. There exists a module $D$ such that $  D \ \oplus E(A/\fp A) ) = C.$ 
Recall that we proved
$\oplus_{\fp \in \Sigma} E(A/\fp) \hookrightarrow \oplus Q$. We
apply  the exact functor  $Q \otimes -$ to it, and deduce
$$Q \otimes (\oplus_{\fp \in \Sigma} R_{ \fp}) \stackrel{\subseteq}\longrightarrow \oplus (Q \otimes Q ) \cong \oplus_{t^2} Q .$$

Also, we know  $Q \otimes ( \oplus R_{ \fp}) \cong \oplus (Q \otimes R_{ \fp}) \cong \oplus Q .$ 
Let us apply this in our previous displayed formula, and deduce the following computation of rank of $Q$-vector spaces: $$\infty=\mid \Sigma\mid=\dim (\oplus_{\fp \in \Sigma} Q ) \leq t^{2}<\infty.$$ This is a contradiction.
\end{proof}

\begin{corollary}
	Let $(R, \fm )$ be a complete local integral domain. The following assertions are equivalent:\begin{enumerate}
			\item[i)]Any torsion-free module of finite rank   is $Q^t\oplus M$ where $M$ is finitely generated.
	\item[ii)]  $\dim R\leq 1$.
	\item[iii)]Any torsion-free module of finite rank  is Matlis reflexive.
	
	\item[iv)] Suppose $S$ is torsion-free and of finite rank. Then $\fm S=S$ iff $S$
	is injective.
	
	 	\end{enumerate}
\end{corollary}

 \begin{proof}We may assume $R$ is not field.
 	
$i) \Rightarrow ii)$: If not, then there is prime ideal $\fp$ such that $0\subsetneqq \fp\subsetneqq \fm$. Let $f\in\fm \setminus \fp$. By the assumption $R_f=Q^t\oplus M$ where $M$ is finitely generated. Since rank $R_f$ is one, and it is not finitely generated we see $R_f=Q$. But, $\fp R_f$ is a nonzero prime ideal. This is a contradiction.

 $ii) \Rightarrow  i):$ This is in \cite[Cor 3]{matcanada}.

 $i) \Rightarrow iii)$: Let  $F$ be torsion free module of finite rank rank. By i) $F=Q^t\oplus M$ where $M$ is finitely generated. By $i) \Rightarrow ii)$, $\dim R=1$.  In view of Proposition \ref{mtq}, $Q^t$ is reflexive. Thanks to Matlis theory, $M$ is reflexive. So, $F$ is reflexive.

 $iii) \Rightarrow  i):$ Since $Q$ is reflexive, and in light of Proposition \ref{mtq}, one has $\dim R=1$.
 By $ii) \Rightarrow  i) $ we get the claim.
 
 $i) \Rightarrow  iv):$ Suppose $S$ is torsion-free and of finite rank and that $\fm S=S$. By i) $S=Q^t\oplus M$ where $M$ is finitely generated. Towards a contradiction, suppose $M\neq 0$. According to Nakayama's lemma,
 $\fm M\neq M$. Consequently, 
 $$\begin{array}{ll}
\fm S
  = \fm(Q^t\oplus M) =\fm Q^t\oplus \fm M  =  Q^t\oplus \fm M \subsetneqq  Q^t\oplus  M =S.
 \end{array}$$ This contradiction says that $M=0$.   So $S=Q^t$, i.e., $S$ is injective.
 
 Clearly, any injective module is divisible.
 
 $iv) \Rightarrow  ii):$
  If not, then there is prime ideal $\fp$ such that $0\subsetneqq \fp\subsetneqq \fm$.
  Let $S:=R_{\fp}$. Then $\fm S=S$, but $S$ is not injective.
 \end{proof}

The following extends \cite[Corollary, page 575]{matcanada}:
\begin{proposition}\label{30}
Let $R$ be a local integral domain and $V$ is a ring between $R$ and $Q$. Let $T$ be a torsion-free $V$-module of rank 1. Then $\Ext_{V}^{i} (Q,T) \cong \Ext_{R}^{i} (Q,T)$. In particular $$\Ext_{R}^1(Q,T) =0 \Leftrightarrow \Ext_{V}^1(Q,T)=0.$$
\end{proposition}
\begin{proof}
Since $T$ is torsion-free and of rank 1 then
 $T \hookrightarrow S^{-1}T \cong Q(V)=Q(R)=Q.$ 
There is an spectral sequence\footnote{one may argue without any use of  spectral sequence.}
$$\E^{ij} := \Ext^i_V ( \Tor_j^R (T,Q), -) \Rightarrow \Ext_{R}^{i+j} ( Q,-)$$
where $-$ is any $V$-module. Since $Q$ is $R$-flat then $\Tor_+^R (-,Q)=0$. So, the spectral sequence collapses. This says
$E^{i0} \cong \Ext_{V}^{i+0} (Q,-)$ and $T \otimes  Q$. Since $R \subseteq T \subseteq Q$ and $Q$ is flat as an $R$-module,
 $Q = R \otimes_R Q \subseteq T \otimes_R Q \subseteq Q \otimes_R Q =Q. $ 
This gives $T \otimes_R Q =Q$,
 so $\Ext_{V}^{i} (Q,-) =\Ext_{R}^{i} (Q,-)$.
 Apply this for $(-):=T$, we get the claim.
\end{proof}

\begin{remark}Let $(R,\fm)$ be a noetherian local ring.  
\begin{enumerate}
	\item[i)]	If $R$ is a domain, then $Q$ is an injective as an $R$-module.

	\item[ii)]	If $E$ is an injective $R$-module, then $H_{\frak{m}}^{i}(E)=0$ for every $i > 0$.

	\item[iii)]  If $\dim (R) =1$, then $Q=R_{x}$ for every $0 \neq x \in \frak{m}$. Indeed, let
	$S:= \lbrace x^{n} \rbrace _{n \geq 1} $
	$$ \Spec (S^{-1} R) =\lbrace S^{-1} \fp \mid \fp \cap S=\emptyset  \emph{ and } \fp \in \Spec(R) \rbrace = \lbrace 0 \rbrace. $$
Since $R_x$ is zero dimensional domain, $R_x$ is a field. Also, $R_x \subseteq Q $ and $Q$ is  the smallest field containing $R$. In sum, $R_x = Q$.\item[iv)]
If $\dim (R) =1$ then the sequence  $0 \to R \to Q \to H_\frak{m}^1 (R) \to 0$ is exact.
\end{enumerate}
\end{remark}
Here, is the higher dimensional  version of \cite[Theorem 1]{matcanada}:
\begin{proposition}\label{31}  Let $(R,\fm)$ be a local integral domain. Then $\dim Q(R)/R =\dim R -1$. In particular,
$Q/R$ is an artinian if and only if $\dim(R) =1$.
\end{proposition}
\begin{proof}
We look at the exact sequence $ 0 \to R \to Q \to Q/R \to 0$ and the induced long exact sequence of local cohomology modules:
$$H_{\frak{m}}^{d-1} (Q) \longrightarrow H_{\frak{m}}^{d-1} (Q/R) \longrightarrow H_{\frak{m}}^{d} (R) \longrightarrow H_{\frak{m}}^{d} (Q)=0$$
By Grothendieck non-vanishing theorem
$H_{\frak{m}}^{d} (R) \neq 0$ and then $H_{\frak{m}}^{d-1} (Q/R) \neq 0$.
According to Grothendieck vanishing theorem, $\dim (Q/R) \geq d-1$.
Since $Q/R$ torsion, it is not of full dimension and then $d-1\leq \dim (Q/R) <d$. Therefore, $\dim (Q/R) =d-1$.

To see the last assertion, 
  first assume that $d > 1$, then $\dim Q/R >0$ and it is not an artinian. Then
Suppose $d=1$. In this case we
apply previous remark to see that the following   $0 \to R \to Q \to H_\frak{m}^1 (R) \to 0,$ is exact. So, $Q/R=H_\frak{m}^d (R)$, which is artinian.
\end{proof}
Suppose $A$ is torsion-free. In \cite[3.2.3]{mat2}, Matlis proved a fundamental isomorphism: $$A\emph{ is cotrsion}\Longleftrightarrow\Hom_R(Q/R,Q/R\otimes_RA)\cong A.$$Let us use the concept of strongly cotrsion, and extend the isomorphism: 	
\begin{proposition}
	For any finitely generated and cotrsion module $A$ we have $$\Hom_R(Q/R,Q/R\otimes_RA)\cong A/ t(A).$$

\end{proposition}
\begin{proof}
	Let $\overline{A}:=A/t(A)$ and we look at $$0\lo t(A)\lo A\lo \overline{A}\lo 0\quad(+)$$ Apply $Q/R\otimes_R-$ to it yields that
	$$  Q/R\otimes_Rt(A)\lo Q/R\otimes_R A\lo Q/R\otimes_R\overline{A}\lo0 \quad(\ast)$$
	Clearly, $Q\otimes_Rt(A)=0$.	Now, we apply $-\otimes_Rt(A)$ to $0\to R\to Q\to Q/R\to 0$ gives us
	 $0  =Q\otimes_Rt(A)\to Q/R\otimes_Rt(A)\to 0 	 .$  In other words, $Q/R\otimes_Rt(A)=0$.
	We put this in $(\ast)$ and see $$Q/R\otimes_R A\cong Q/R\otimes_R\overline{A}\quad(\dagger)$$

	Now, recall that finitely generated module is torsion-bounded. This gives
	that $t(A)$ is  torsion of bounded order.  According to \cite[page 4]{mat2}
we have the vanishing property	$\Ext^\ast_R(Q,t(A))=0$. By the terminology of Definition \ref{1},  
	$t(A)$ is strongly cotorsion. Recall that  $A$ is  cotorsion. In particular 
we are in the situation of \cite[Lemma 1.1.3]{mat2}.	Now, we apply it along with  $(+)$ to deduce that $\overline{A}$
	is cotorsion.
Let us apply the mentioned result of Matlis, in the torsion-free case, to deduce:$$\Hom_R(Q/R,Q/R\otimes_RA)\stackrel{(\dagger)}\cong\Hom_R(Q/R,Q/R\otimes_R\overline{A})\cong \overline{A}.$$
This is what we want to prove.
\end{proof}

\begin{remark}An R-module $B$ contains a unique largest h-divisible submodule, denoted by
$\h(B)$. For any torsion  module $B$ Matlis proved in \cite[3.1]{mat2}
that $Q/R\otimes_R\Hom_R(Q/R,B)=\h(B)$. Here, we remove the torsion assumption. Indeed,
it is easy to see  $Q/R\otimes_R\Hom_R(Q/R,M)=\h(t(M)),$ 
where $M$ is any module. To see this, it is enough to note that
$\Hom(Q/R,M/t(M))=0$ and apply \cite[3.1]{mat2}.
\end{remark}
The following is a higher dimensional
version of \cite[Theorem 13.2]{E}:

\begin{proposition}\label{312}  Let $(R,\fm)$ be a $d$-dimensional Cohen-Macaulay\footnote{recall that the Cohen-Macaulay assumption is automatically holds if $d=1$.}  integral domain such that $\fm$ generated by $d+1$ elements. Then $\id_R (Q(R)/R) =d -1$. In fact, $R$ is complete-intersection. 
\end{proposition}
\begin{proof}
	By \cite[Ex. 21.2]{mat}
	$R$ is complete-intersection. In particular, $\id(R)=d$. Recall that $Q$ is injective.  By $ 0 \to R \to Q \to Q/R \to 0$ we get to the desired claim.
\end{proof}
 Let us connect to Gorenstein-homology.
As a sample, by $\Gpd_R(-)$ we mean the Gorenstein projective dimension of a module $(-)$. For its definition, e.g., see  \cite[Definition 2.1]{f}. The following is Gorenstein analogue of \cite[Cor. 4]{matcanada}:
\begin{observation}
  Let $(R,\fm)$ be a $1$-dimensional complete local  integral domain and $S$  be torsion-free and of finite rank. Then $\Gpd_R(S)<\infty$ iff there is  a totally reflexive and finitely generated module $M$ such that $S=Q^t\oplus M$.
\end{observation}

\begin{proof}
First, assume that $\Gpd_R(S)<\infty$. It follows by \cite[Cor. 3]{matcanada}
that there are $t\geq 0$ and a finitely generated module $M$ such that $S=Q^t\oplus M$.
 As $M$ is a directed summand of $S$,
 $\Gpd(M)<\infty$. Since $M$ is 
finitely generated, its $G$-dimension is finite. We may assume $M\neq 0$. Since it is torsion-free, $\depth(M)>0$, and so $0<\depth(M)\leq \dim(M)=1$. By Auslander-Bridger formula (see \cite[Theorem 1.25]{f})
$$1+\Gdim_R(M)=\depth(M)+\Gdim_R(M)=\depth(R)=1,$$
$M$ is totally reflexive.

In order to see the converse part, it is enough to recall that
$\Gpd_R(Q)\leq \pd_R(Q)=1$ (see \cite[Proposition 2.12]{f}) and apply \cite[Cor. 3]{matcanada}.
\end{proof}

The following extends a result of Matlis \cite[Theorem 41]{matc} by replacing reflexive with totally  reflexive:

\begin{corollary}\label{312to}(Matlis-Bass)  Let $(R,\fm)$ be a $1$-dimensional local  integral domain such that $\fm$ generated by $2$ elements. Then any ideal of $R$  is totally reflexive.
\end{corollary}

\begin{proof}
In the light of  \ref{312}, $R$ is Gorenstein. Let $I$ be a nontrivial ideal. Following  Auslander-Bridger formula we have
$$ 1+\Gdim_R(I)=\depth(I)+\Gdim_R(I)=\depth(R)=1,$$
and so, $\Gdim_R(I)=0$. In other words,  $I$ is totally reflexive.
\end{proof}

\section{From Matlis to Vasconcelos}

First, we reprove the following beautiful result of Vasconcelos \cite[1.2]{vo}:

\begin{theorem}
	Let $R$ be a noetherian ring and $M$ be finitely generated. Then any surjective $f:M\to M$ is an isomorphism.
\end{theorem}

\begin{proof}
The properties under consideration are local. Then we assume $R$ is local. So, $d:=\dim R<\infty$. The proof is by induction. The proof is by induction on $d$. Suppose  $d=0$.
Then any finitely generated module is of finite length. Let $K:=\Ker f$. Taking the length from $o\to K\to M \to M\to M\to 0$, we see $\ell(K)=0$, and so $K=0$. Then we may assume $d>0$ and the desired claim holds for all ring of dimension less than $d$. In particular, $\ell(K)<\infty$. Since $f$ is surjective, there is $K_1\subseteq M$ such that $ 0\lo \ker(f \upharpoonright)\lo K_1\stackrel{f \upharpoonright} \lo K\lo 0$ is exact. Since $\ker(f \upharpoonright)\subseteq K$,   we have $\ell(\ker(f \upharpoonright))<\infty$. So, $\ell(K_1)<\infty$. By adding $K$ if needed, we may assume in addition that $K\subseteq  k_1$. By repeating this, there are finite length module $K\subseteq K_n\subseteq M$ such that $f(K_{n+1})=K_n$. Let $L:=\langle K_n: n\in\mathbb{N}\rangle\subseteq M$. It is supported at the maximal ideal and is noetherian, because $M$ is finitely generated. So, $\ell(L)<\infty$.
Let $\overline{f}:L\to f(L)$ be the surjective map defined by $f$. Since $f(K_n)=K_{n+1}$ we see $f(L)=L$. By taking length, $\ell(\ker(\overline{f}))=0$. Thus,
$\ker(\overline{f})=0$.
In particular, $K=0$, as claimed.
\end{proof}
The following is the higher dimensional version of \cite[5.16]{E}:
\begin{corollary}
Let $R$ be a local ring and $A$ be artinian. Then any injective $f:A\to A$ is an isomorphism.
\end{corollary}
\begin{proof}
	We may assume that $R$ is complete. There is an exact sequence $$0\lo A\stackrel{f}\lo A\lo C:=\Coker(f)\lo 0$$Matlis' dualizing,
	gives  $0\to C^{ \vee} \lo A^{ \vee}\stackrel{g}\to A^{ \vee}\to 0$ Since $A^{ \vee}$ is finitely generated, $C^{ \vee}$ is as well. These class is Matlis reflexive. In view the above theorem, $g$ is an isomorphism. So, $C^{ \vee}=0$. Then, by Matlis dual theory, $C\cong C^{\vee \vee}=0,$ as claimed.
\end{proof}

Suppose $(R,\fm)$ is a 1-dimensional domain.
In \cite[Theorem 7.1]{E} Matlis proved that
any nonzero map $Q/R\to Q/R$ is surjective iff $R$ is analytically unramified. In this 1-dimensional case, recall that $Q/R=H_\frak{m}^d (R)$. Here, is the higher version:

\begin{proposition}\label{coh}
	Let $(R,\fm)$ be a d-dimensional   local  Cohen-Macaulay ring. Then any nonzero $f:H^d_\fm(R)\to H^d_\fm(R)$ is surjective iff $R$ is analytically irreducible.
\end{proposition}
\begin{proof}Recall that any nonzero $f:H^d_\fm(R)\to H^d_\fm(R)$ is surjective. This property behaves well with respect to completion, as it is artinian.
 We may assume  that $R$ is complete, and w3 are going to show that $R$ is an integral domain. The assumptions grantee that the canonical module $\omega_R$ exists. Let $r$ be nonzero. By $\mu_r:\omega_R\to \omega_R$ we mean the multiplication by $r$. Since $\Hom(\mu_r,E(R/\fm)) :H^d_\fm(R)\stackrel{r}\lo H^d_\fm(R)$ is nonzero, by the assumption it is surjective. The following diagram

 $$
 \begin{CD}
 0@>>> H^d_\fm(R)^{ \vee} @> r >> H^d_\fm(R)^{ \vee}  \\
 @.\cong@AAA\cong@AAA          \\
  @. \omega_R @>\mu_r >> \omega_R , \\
 \end{CD}
 $$
 say that $r$ is $\omega_R$-regular. In other words, $r\notin \Zd(\omega_R):=\bigcup_{\fp\in\Ass(\omega_R)}\fp$. Now, recall that
 \begin{enumerate}
 	\item[i)] $\Ass(\omega_R)=\Min(\omega_R)=\min(\Supp(\omega_R))$,

 	\item[ii)]	$\Supp(\omega_R) =\Spec(R)$.

 \end{enumerate}
We apply this and see$$\begin{array}{ll}
\Ass(\omega_R)&=\min(\Supp(\omega_R)) = \min(\Spec(R)) =\min(R) =\Ass(R).
\end{array}$$
This in turn implies that $$\begin{array}{ll}
r\notin \Zd(\omega_R) =\bigcup_{\fp\in\Ass(\omega_R)}\fp =\bigcup_{\fp\in\Ass(R)}\fp = \Zd(R).
\end{array}$$
In sum, $r$ is $R$-regular, and so $R$ is a domain. Conversely,
suppose $R$ is domain, and look
any nonzero $f:H^d_\fm(R)\to H^d_\fm(R)$.
Then $$f^{ \vee}\in (\Hom_R(H^d_\fm(R)^{ \vee}, H^d_\fm(R)^{ \vee})=\Hom_R(\omega_R,\omega_R)=R.$$
There is a nonzero  $r\in R$ such that $f^{ \vee}$
is a multiplication by $r$. Since $R$ is domain,multiplication by $r$ is one-to-one over $\omega_R$. In particular, $f^{\vee \vee}:H^d_\fm(R)\to H^d_\fm(R)$ is surjective.
Since  $f=f^{\vee \vee}$, the desired claim is now clear.
\end{proof}

\begin{example}
	Take $R=k[[x,y]]/(xy)$. It is complete, reduced, one-dimensional, and
	Gorenstein, so it is analytically unramified. Let $H=H^1_\m(R)$ and let
	$(-)^\vee$ denote Matlis duality. Since $H^\vee\cong R$, the nonzero
	endomorphism $x:H\to H$ is dual to $x:R\to R$. The latter is nonzero
	but is not injective, since it kills $y\ne0$. Thus $x:H\to H$ is not
	surjective.
	
\end{example}
\begin{definition}Let $(R,\fm)$ be local.
	We say an $R$-module $E$ is highly divisible if for any system of parameter $x_1,\ldots,x_d$ of $R$, the multiplication maps  $\Ann_E(x_1,\ldots, x_i)\stackrel{x_{i+1}}\lo \Ann_E(x_1,\ldots, x_i)\lo 0$ are surjective.
\end{definition}

The following is the higher dimensional version of \cite[7.1]{E}:

\begin{corollary}
	 Let $(R,\fm)$ be a d-dimensional complete  local Cohen-Macaulay domain.
	 Then $H^d_\fm(R)$ has no nontrivial highly divisible submodule.
\end{corollary}

\begin{proof}
Let $A$ be a	nontrivial highly divisible submodule of $H^d_\fm(R)$. Here, we use the concept of consequence \cite[Page 97, Definition]{matne}.
Following definition, any $R$-sequence
is $A$-consequence. Since $H^d_\fm(R)$ is artinian,
the module $A$ is as well. These allow us to apply \cite[Lemma 6.6]{matne}
 to conclude that the natural evaluation map $$H^d_\fm(R)\otimes_R\Hom_R(H^d_\fm(R),A) \lo A\quad(\ast),$$is an isomorphism. Since 
$A\neq 0$, and by $(\ast)$ we deduce that
$\Hom(H^d_\fm(R),A) 	\neq 0$. Let $f:H^d_\fm(R)\to A$ be any nonzero map. Now, we look at the following map  $g:=H^d_\fm(R)\stackrel{f}\lo A\stackrel{\subsetneqq}\lo H^d_\fm(R).$ Clearly,
$g$ is nonzero and it is not surjective. This is in contradiction with  Proposition \ref{coh}. So, $A$ should be trivial.
\end{proof}
In \cite[Cor 8]{matcanada} Matlis proved over 1-dimensional domain that
$\Hom_R(Q/B,Q/R)\neq0$ where $B\subsetneqq Q$.
Here is the higher\footnote{Recall from Discussion \ref{dis} that in the 1-dimension case we have $Q/B\cong H^d_\fm(B)$.}:
\begin{theorem}
\label{fff} Let $(R,\fm)$ be  a local Cohen-Macaulay ring and $B$ be torsionless as an $R$-module. Then $\Hom_R(H^d_\fm(B),H^d_\fm(R))\neq 0$.
\end{theorem}

\begin{proof}
	First, 
	$$\begin{array}{ll}
	\Hom_R(H^d_\fm(B),H^d_\fm(R))&=\Hom_R(H^d_\fm(R)\otimes_RB,H^d_\fm(R))\\
	&= \Hom_R(B,\Hom_R(H^d_\fm(R),H^d_\fm(R))=\Hom_R(B,\widehat{R}).
	\end{array}$$
Thus, things are reduced to show $\Hom_R(B,\widehat{R})\neq 0$. Let $\oplus R\to B^\ast\to 0$ be a part of free resolution.
Since $B$ is torsionless as an $R$-module, we have$$\begin{array}{ll}
B \subseteq B^{\ast\ast} \subseteq \Hom_R(\oplus R,R) =\prod\Hom_R(R,R) =\prod R \subseteq\prod \widehat{R}.
\end{array}$$
Since  $0\neq \Hom_R(B,\prod\widehat{R})\cong \prod \Hom_R(B, \widehat{R}), $ we see $\Hom_R(B,\widehat{R})\neq 0$. The claim follows.
\end{proof}

Also, he proved over 1-dimensional rings that $\Hom_R(D,Q/R)\neq0$ where $D$ is divisible, see \cite[Corollary 8.10]{E}. Here, is the higher dimensional version:

\begin{proposition}
Let $(R,\fm)$ be a complete  Cohen-Macaulay domain and $D$ be nonzero and divisible as an $R$-module. Then $\Hom_R(D,H^d_\fm(R))\neq 0$.
\end{proposition}

\begin{proof}
Recall from Matlis duality, dual of divisible is torsion-free. 	We have $$\begin{array}{ll}
\Hom_R(D,H^d_\fm(R))&=	\Hom_R(D,\omega_R^{ \vee}) = \Hom_R(D\otimes_R\omega_R,E_R(R/\fm))\\
	&=\Hom_R(\omega_R\otimes_R D,E_R(R/\fm)) =\Hom_R(\omega_R,D^{ \vee})
	\\
	&=
	\Hom_R(\omega_R,\emph{torsion-free}).
	\end{array}$$Since $R$ is generically Gorenstein,
 $\omega_R\stackrel{\subseteq}\lo R\stackrel{\subseteq}\lo \emph{torsion-free}.$ 
	The claim is now clear.
\end{proof}
Here, we present a sample that a Matlis result can not  extend to higher. First, recall:

\begin{fact}
	(Matlis, \cite[Corollary 8.10]{E})
	Let $(R,\fm)$ be a 1-dimensional domain and $B\varsubsetneq Q$. Then $\Hom_R(B, \widehat{R})\neq 0$.
\end{fact}

How can extend this to higher? Here, is the answer:

\begin{theorem}\label{nex}
 Let $(R,\fm)$ be a  domain which is not a field. The following are equivalent:
	\begin{enumerate}
		\item[i)] $\Hom_R(B, \widehat{R})\neq 0$ for all $B\varsubsetneq Q$.
	 \item[ii)]
		  $d:=\dim(R)=1$.  	\end{enumerate}
	
\end{theorem}

\begin{proof}
$i) \Rightarrow ii)$: If not, then $d>1$, and look at $B:=R_x$ for some $0\neq x\in \fm$. It is easy to see that $B\varsubsetneq Q$, because $d>1$. By assumption, there is a nonzero $f\in\Hom_R(B, \widehat{R})$. Let $I:=\im(f)$, and let $\alpha\in I$ be nonzero. There are $r\in R$ and $n\in\mathbb{N}$ such that $\alpha=f(r/x^n)$.	In view of the equation  $xf(r/x^{n+1})=f(rx/x^{n+1})=f(r/x^n)=\alpha,$ we see $\alpha/x\in I$. Now, we look at 
 $(\alpha,\alpha/x,\alpha/x^2,\ldots ).$ 
We leave  to the reader to see that  $\alpha/x^{n+1} \notin(\alpha,\alpha/x,\ldots,\alpha/x^n ).$ 	Since $\widehat{R}$ is noetherian we get to a contradiction. 

$ii) \Rightarrow  i)$: This is in \cite[Corollary 8.10]{E}.
\end{proof}

\section{Grothendieck's localization problem}
Rings in this section are not necessarily noetherian.
  By $\gldim(R)$ we mean the global dimension.
	Matlis proved that $\gldim(R) \leq \gldim( \overset{\sim} R)$, and with equality in a special case.
For simplicity, let us recall non-noetherian grade:

\begin{definition} Let $\fa$ be an ideal of a ring $R$ and $M$ an
	$R$-module. Take $\Sigma$ be the family of all finitely generated
	subideals $\fb$ of $\fa$. Here, $\inf$ and $\sup$ are formed in
	$\mathbb{Z} \cup \{\pm\infty\}$ with the convention that $\inf
	\emptyset=+ \infty$ and $\sup \emptyset=- \infty$.
	
	(i) In order to give the definition of Koszul grade when $\fa$ is
	finitely generated  by a generating set
	$\underline{x}:=x_{1},\cdots, x_{r}$, we first denote the Koszul
	complex related to $\underline{x}$ by
	$\mathbb{K}_{\bullet}(\underline{x})$. Koszul grade of $\fa$ on $M$
	is defined by
	$$\Kgrade_R(\fa,M):=\inf\{i \in\mathbb{N}\cup\{0\} | H^{i}(\Hom_R(
	\mathbb{K}_{\bullet}(\underline{x}), M)) \neq0\}.$$ Note that by
	\cite[Corollary 1.6.22]{BH} and \cite[Proposition 1.6.10 (d)]{BH},
	this does not depend on the choice of generating sets of $\fa$. For
	an ideal $\fa$ (not necessarily finitely generated), Koszul grade of
	$\fa$ on $M$ can be defined by
	$$\Kgrade_R(\fa,M):=\sup\{\Kgrade_R(\fb,M):\fb\in\Sigma\}.$$ By using
	\cite[Proposition 9.1.2 (f)]{BH}, this definition coincides with the
	original definition for finitely generated ideals.
	
	(ii) A finite sequence $\underline{x}:=x_{1},\cdots,x_{r}$ of
	elements of $R$ is called weak regular sequence on $M$ if $x_i$ is a
	nonzero-divisor on $M/(x_1,\cdots, x_{i-1})M$ for $i=1,\cdots,r$. If
	in addition $M\neq (\underline{x})M$, $\underline{x}$ is called
	regular sequence on $M$. The classical grade  of $\fa$ on $M$,
	denoted by $\grade_R(\fa,M)$, is defined to the supremum of the
	lengths of all weak regular sequences on $M$ contained in $\fa$.

	(iii) The Hochster's polynomial grade of $\fa$ on M is
	defined by
	$$\pgrade_R(\fa,M):=\underset{m\rightarrow\infty}{\lim}
	\grade_{R[t_1, \cdots,t_m]}(\fa R[t_1, \cdots,t_m],R[t_1,,
	\cdots,t_m]\otimes_R M).$$

\end{definition}

\begin{definition} A quasi-local ring $(A,\fm)$ is called Cohen-Macaulay if $\pgrade(\fm,A)=\dim(A)$ where $A$ is any commutative local ring.\end{definition}

\begin{proposition}\label{39}
Let $(R,\frak{m})$  be a  quasi-local domain. If $\overset{\sim} R$ is Cohen-Macaulay,  then $R$ is Cohen-Macaulay.
\end{proposition}
\begin{proof}
Recall that $\overset{\sim} R = \End_R (Q/R) $ is commutative local ring with a unique maximal ideal $\fm \overset{\sim} R $. Also $R \overset{\varphi}{\longrightarrow} \overset{\sim} R $ is faithfully flat where $\varphi$ defined as $\varphi (r) \in \overset{\sim} R  = \End_R(Q/R)$.
By \cite[5.10.4]{mat2}, If $p \subseteq R $ is prime then $p \overset{\sim} R$ is prime in $\overset{\sim} R$.
So if
 $  P_{1} \subsetneqq P_{2} \subsetneqq \ldots \subsetneqq P_{t} \subseteq R$ 
is chain of prime, then
$$  P_{1}\overset{\sim} R \subsetneqq P_{2}\overset{\sim} R \subsetneqq ... \subsetneqq P_{t} \overset{\sim} R \subseteq \overset{\sim} R$$
is chain of prime in $\overset{\sim} R$. In other words
 $\dim(\overset{\sim} R) \geq \dim(R).$
Now
$$\dim(R) \geq \pgrade(\fm,R) =\pgrade(\fm\overset{\sim} R, \overset{\sim} R) =\dim(\overset{\sim} R) \geq \dim(R).$$
So, $R$ is Cohen-Macaulay.
\end{proof}

\begin{proposition}\label{40}
Let $R$ be an integral domain such that $ \pgrade (I ,\overset{\sim} R) = \Ht( I )$  for all ideals of $I \triangleleft \overset{\sim} R$. Then $\pgrade (J,R) =\Ht(J) $ for all $J \triangleleft R$.
\end{proposition}

\begin{proof}
Let $J$ be an ideal of $R$. In view of \cite[Proposistion 2.2(vi)]{AT} there exists $P \in \Spec(R)$ such that $\pgrade (J,R)= \pgrade (P,R)$.
Let
$$ 0 \subsetneqq P_{0} \subsetneqq P_{1} \subsetneqq \ldots \subsetneqq P_{t} \subseteq P$$
be a chain of primes. According to \cite[5.10.4]{mat2} the following chain
$$ 0 \subsetneqq P_{0}\overset{\sim} R \subsetneqq P_{1}\overset{\sim} R \subsetneqq \ldots \subsetneqq P_{t}\overset{\sim} R \subseteq P\overset{\sim} R$$
is a chain of prime ideals.
By \cite[5.10]{mat2}, $P_{0}\overset{\sim} R$ is not minimal (as $\overset{\sim} R/P_{0}\cong R/P_0$ is not field). There exists $Q \subsetneqq P_{0}\overset{\sim} R$ such that $\Ht(P) \leq \Ht(P\overset{\sim} R)$ (the ideal $Q$ may be zero or not).
$$\begin{array}{ll}
\pgrade (P,R)  \leq \Ht(P)  
 \leq \Ht(P\overset{\sim} R)  =\pgrade(P\overset{\sim} R,\overset{\sim} R).\\
\end{array}$$
So $\pgrade (P,R) = \Ht(P)$.
Recall that
$$\begin{array}{ll}
\Ht(J)  \leq \Ht(P)  \leq \Ht(P\overset{\sim} R) =\pgrade(P\overset{\sim} R,\overset{\sim} R).\\
\end{array}$$
Then $\Ht(J) = \pgrade (J)$.\\
\end{proof}

 Recall that Matlis proved $\gdim(R) \leq \gdim(\overset{\sim} R)$. This may be not sharp. He introduced closed rings and showed that in this case the equality holds \cite[Thm 4.6]{d}, in particular over closed rings, we have $R$ is regular if and only if $\overset{\sim} R$ is regular. Here is the Cohen-Macaulay version:

\begin{proposition}\label{41}
Let $(R,\fm)$ be a closed quasilocal domain. Then $R$ is Cohen-Macaulay if and only if $\overset{\sim} R$ is Cohen-Macaulay.
\end{proposition}
\begin{proof}
Suppose $R$ is Cohen-Macaulay, then
$\pdepth (R) = \dim(R)$.
Since $\pdepth (\fm, R) = \pdepth (\fm\overset{\sim} R, \overset{\sim} R)$ and
 $\pdepth (\fm, R) = \dim(R) \leq \dim \overset{\sim} R.$ 
If $\dim (R) = \infty $, then $\dim (\overset{\sim} R) = \infty $ and $\pdepth (\fm\overset{\sim} R)= \infty $. Then $\overset{\sim} R$ is Cohen-Macaulay.
Suppose $\dim (R) <\infty $.
Let
$$P_{0} \subsetneqq P_{1} \subsetneqq P_{2} \subsetneqq P_{3} \subsetneqq \ldots \subsetneqq P \subseteq \overset{\sim} R .$$
It is easy to see $Q_i: = P_i \cap R $ is a prime in $R$. In fact
$$Q_{0} \subsetneqq Q_{1} \subsetneqq Q_{2} \subsetneqq Q_{3} \subsetneqq \ldots \subsetneqq Q_l \subseteq \overset{\sim} R, $$
is a chain of prime ideal.
Suppose on the way of contradiction that $Q_{i} = Q_{j}$ for some $i,j$ where $i<j$.
In the light of \cite[ Thm 4.5]{d}, $Q_{i}\overset{\sim} R = P_{i}$ and $Q_{j}\overset{\sim} R = P_{j}$, so $P_{i} = P_{j}$.
But $P_{i} \subsetneqq P_{j}$. This contradiction shows that
 $Q_{0} \subsetneqq Q_{1} \subsetneqq \ldots \subsetneqq Q_{l}, $
so $\dim (R) = l = \dim(\overset{\sim} R) $.
Now recall
$$\begin{array}{ll}
\dim(R) &= \dim(\overset{\sim} R)  = \pdepth(\fm\overset{\sim} R, \overset{\sim} R) =\pdepth (\fm,R). 
\end{array}$$
By definition, $R$ is Cohen-Macaulay.
\end{proof}
The above proof shows:
\begin{corollary}\label{42}
If $R$ is closed, then $\dim(R)= \dim( \overset{\sim} R)$.
\end{corollary}

\begin{definition}
	We say a quasi-local ring $R$ is Gorenstein if $\id_R(R)<\infty$.
\end{definition}
 \begin{proposition}\label{411}
 	Let $(R,\fm)$ be a closed quasilocal domain. Then $R$ is Gorenstein if and only if $\overset{\sim} R$ is Gorenstein.
 \end{proposition}
 \begin{proof}
 We use the following, where $A\in\{\overset{\sim} R,R\}$;
 $\id_A(M)=\sup\{i:\Ext_A^i(torsion, M)\neq 0\}.$ In view of \cite[5.5.1]{mat2}	$$\lbrace\emph{torsion modules over }R\rbrace \rightleftharpoons \lbrace \emph{torsion modules over } \overset{\sim} R \rbrace.$$
Let $T$ be torsion. Recall
 $$\begin{array}{ll}
\Ext_R^i(T, \overset{\sim} R) &\stackrel{1}= \Ext_{ \overset{\sim}R}^i(T, \overset{\sim} R)  \stackrel{}= \Ext_{ \overset{\sim}R}^i(T, \overset{\sim} R\otimes R)\\
 &\stackrel{2}=\Ext_{ R}^i(T, \overset{\sim} R\otimes R)\stackrel{3}=\Ext_{ R}^i(T, R),\\
 \end{array}$$
 where  1)  is in \cite[5.5.1]{mat2}, 2) 	is in \cite[5.5.2]{mat2} and
 3)    is in \cite[5.5.3]{mat2}.
 So,  $\Ext_R^i(T, \overset{\sim} R)\cong \Ext_{ \overset{\sim}R}^i(T, \overset{\sim} R).$
We proved that 	$\id_R(R)=\id_{ \overset{\sim}R}(\overset{\sim}R)$. The desired  claim is clear by this.
 	\end{proof}
	A natural generalization of noetherian rings is the coherent ring:
	\begin{definition}\label{44}
	$R$ is called coherent if any finitely generated submodule of a finitely presented module is again finitely presented.
\end{definition}
The book of Glaz \cite{G} is a useful reference.

\begin{notation}
By $\fd_R(-)$ we mean the flat dimension of $(-)$. By weak dimension of $R$ we mean  
	$$\wdim (R) := \sup \lbrace \fd_R (M) \mid \emph{ M is an R-module}\rbrace.$$
\end{notation}

Matlis proved is \cite[9.2]{mat2} that
 $\wdim (R) \leq \wdim( \overset{\sim} R)$. Is this sharp?  Here, is the answer:
\begin{proposition}\label{45}
	Let $(R, \fm)$ be a quasi-local  integral domain such that $ \overset{\sim} R$ is coherent. Then  $\wdim(R)=\wdim( \overset{\sim} R)$.
\end{proposition}
\begin{proof}
	We know that by the work of Matlis \cite[9.2]{mat2} that
	 $\wdim (R) \leq \wdim( \overset{\sim} R)$.
	Then, without loss of generality we may and do assume that $\wdim(R) <\infty $.
In the light of \cite[2.4.5]{G} we see that $R$ is coherent. This allows us to apply \cite[Thm 2.6]{AT} and deduce that
 $\wdim(R) = \pdepth (R) = \pdepth ( \overset{\sim} R).$ 
	By \cite[9.2.1]{mat2};  $\fd (\overset{\sim} R / \fm \overset{\sim} R ) = \fd ( R/ \fm) \leq \wdim(R) < \infty.$ Recall that $\pdepth(\overset{\sim} R)$ is finite.
	For every $\overset{\sim} R$-module $M$,  $\Tor_i (\overset{\sim} R / \fm \overset{\sim} R , M)=0$  for all $ i> \pdepth(\overset{\sim} R).$ 
	Thanks to \cite[2.5.9]{G}, we know
 $\pd(M) \leq \pdepth(\overset{\sim} R).$ Let us apply this to see
	$$\begin{array}{ll}
	\wdim(\overset{\sim} R) &= \sup \lbrace \pd(M) \mid M \emph {is finitely presented  as an }\overset{\sim}R\emph {-module} \rbrace\\
	&\leq \pdepth( \overset{\sim} R) =\pdepth(R) =\wdim(R) \leq \wdim(\overset{\sim}R). 
	\end{array}$$
The claim is now clear.
\end{proof}

\begin{remark}
	One may repeat the results of this section, by applying different approach
to non-noetherian Cohen-Macaulay (Gorenstein, regular and et cetera) rings. \end{remark}
\section{An application: computing $\pd_R(\widehat{R})$}

Let $(R,\fm)$ be a noetherian local ring. Matlis proved that $\pd_R( \overset{\sim} R )\leq \pd_R(Q)$.

\subsection{Bounding $\pd_R(\widehat{R})$}

\begin{fact}(Matlis).\label{hom}
Suppose $R$ is not complete in $R$-topology. Then $\Hom_R(\overset{\sim} R,R)	=0$.
\end{fact}
\begin{fact}\label{fin}(See \cite{GR}).
	Projective dimension of any flat module is finite.
\end{fact}

We need the following:
\begin{corollary}\label{nfree}
	Suppose $R$ is not complete in $\fm$-adic topology. Then $\Hom_R(\widehat{R},R)	=0$. In particular,
	
	\begin{enumerate}
		\item[i)] $\widehat{R}$ is not free as an $R$-module.
		
		\item[ii)]  $\widehat{R}$ is not finitely generated as an $R$-module.
	\end{enumerate}
\end{corollary}	

\begin{proof}
	Clear by Fact \ref{hom}.
\end{proof}
\begin{hypothesis}
	$R \neq  \overset{\sim} R $.
	\end{hypothesis}
	
\begin{proposition}\label{43}
 One has
 $\pd_R(Q) = \pd_R( \overset{\sim} R )\leq \pd_R(\widehat{R}).$
\end{proposition}
\begin{proof}
By \cite[5.2]{mat2} there is an exact sequence
$$0 \longrightarrow R \longrightarrow  \overset{\sim} R \longrightarrow \Ext^1_R (Q,R) \longrightarrow 0\quad(\ast)$$
If $R \neq  \overset{\sim} R $ then there exists $t >0$ such that $\Ext^1_R (Q,R)= \oplus_t Q$.
Let $p := \pd_R(Q) <\infty $ (see fact \ref{fin}). We know by the work of Matlis that  $\pd_R ( \overset{\sim} R) \stackrel{(+)}\leq \pd_R(Q)$.

Step 1: If $p=1$ then $\pd_R ( \overset{\sim} R ) \leq \pd_R(Q) =1$.
Indeed,  $ 0\lneqq \pd _R( \overset{\sim} R ) \leq \pd _R(Q) = 1$  and so $\pd_R ( \overset{\sim} R )=1$\\
Step 2: $\pd_R (Q)= \pd_R ( \overset{\sim} R ) $ and $p>1$.
Indeed, there exists an $R$-module $M$ such that $\Ext^p _R(Q,M)\neq 0 $.
From $(\ast)$ th following exact sequence derives:
$$\Ext^{p-1}_R (R,M) \longrightarrow \sqcap \Ext^p_R (Q,M) \longrightarrow \Ext^p_R ( \overset{\sim} R,M) \longrightarrow \Ext^p_R (R,M).$$
As $\Ext^p_R ( \overset{\sim} R ,M) \neq 0$ and that $p-1>0$ we observe that  $\pd_R ( \overset{\sim} R ) \geq p=\pd_R(Q).$  Combine this along with $(+)$ we have  $\pd_R ( \overset{\sim} R )= \pd_R (Q) .$

Step 3:  $\pd _R(\widehat{R}/R ) \geq \pd_R ( \overset{\sim} R /R ) \geq \pd_R ( \overset{\sim} R)$.
Indeed, we use the formula of Matlis \cite[7.9]{mat2}:
$ \widehat{R}/R  \cong \widehat{R}/\overset{\sim} R  \oplus  \overset{\sim} R / R.$ 
This gives us $\pd_R (\widehat{R}/R ) \geq \pd _R( \overset{\sim} R /R )$.
We look at
 $0 \longrightarrow R \longrightarrow  \overset{\sim} R \longrightarrow  \overset{\sim} R/R \longrightarrow 0. $ 
Let $j>0$. In the light of the following exact sequence
 $\Ext ^j_R ( \overset{\sim} R/R , -) \to \Ext^j _R( \overset{\sim} R , -) \to \Ext^j_R (R,-)=0,$ 
we observe that  $\pd_R ( \overset{\sim} R /R) \geq \pd_R ( \overset{\sim} R)$ and so $\pd _R(\widehat{R}/R ) \geq \pd_R (\widehat{R})$.

Step 4:
$\pd_R (\widehat{R})= \pd_R (\widehat{R}/R)$.
To see this,
let $j_0>0$, and look at
 $0 \longrightarrow R \longrightarrow \widehat{R} \longrightarrow \widehat{R}/R \longrightarrow 0.$ 
This induces
 $\Ext^{j_0} _R(\widehat{R}/R , -) \longrightarrow \Ext^{j_0}_R (\widehat{R}, -) \longrightarrow \Ext^{j_0}_R (R,-)=0.$ 
 We claim that $\widehat{R}/R$ is not free. If not, then we look at
 $ 0 \longrightarrow R \longrightarrow \widehat{R} \longrightarrow \widehat{R}/R \longrightarrow 0$ 
splits and $\widehat{R} = R \oplus \widehat{R}/R$ be free. This is in contradiction with Corollary \ref{nfree}. In sum, $\widehat{R}/R$ is not free.
Let $j_1> \pd_R (\widehat{R}/R)$. Then
$\Ext^{j_1} _R(\widehat{R},-) =0$ and so $\pd_R (\widehat{R}) \leq \pd_R(\widehat{R}/R)$.
Now, let $j= \pd_R (\widehat{R}/R)$, then  $j>0$ and there exists $R$-module $M$ such that $\Ext^j_R (\widehat{R}/R , M) \neq 0$. There are two possibilities $j=1$ or $j>2$.
Suppose first that $j>2$.  In view of
$$0=\Ext^{j-1}_R (R,M) \longrightarrow \Ext^j_R ( \widehat{R}/R ,M) \longrightarrow \Ext^j_R( \widehat{R} , M) \longrightarrow \Ext^j _R(R,M)=0, $$
$\Ext^j_R( \widehat{R} , M)\cong \Ext^j_R (\widehat{R}/R , M)  \neq 0$.
We conclude that
 $\pd_R(\widehat{R}) \geq j= \pd_R (\widehat{R}/R) \geq \pd_R(\widehat{R}).$  Consequently,
 $\pd_R (\widehat{R}) = \pd_R (\widehat{R}/R).$ 
It remains to deal with $j=1$. Suppose on the way of contradiction that the claim $\pd_R (\widehat{R}) = \pd _R(\widehat{R}/R)$ is not true. We proved that
 $0 < \pd_R(\widehat{R}) \leq \pd_R(\widehat{R}/R) =j =1, $
and so $\widehat{R}$ is free.
This is in contradiction with  Corollary \ref{nfree}.
\end{proof}

Let us record the following:

\begin{corollary}(Matlis)
	Suppose  $R$ is not complete with respect to $\fm$-adic topology. Then $\widehat{R}/R$ is flat. In particular, $0\in \Ass(\widehat{R}/R)$. The same result holds for $\widehat{R}/\overset{\sim} R$ and  $\overset{\sim} R / R$, provided they are nonzero.
\end{corollary}

\subsection{Explicit computations of $\pd_R(\widehat{R})$}
In this subsection, and in order to control the size of rings, we assume   basic of cardinal arithmetic.

\begin{discussion}
	\begin{enumerate}
	\item[i)]Recall that a subset $C$ of a linearly ordered set $P$ is called
	\emph{coinitial} if, for every $x\in P$, there exists $c\in C$
	such that $c\leq x$. The coinitiality of $P$, denoted by
	$\operatorname{ci}(P)$, is the least carnality of a coinitial
	subset of $P$.
	A subset $C$ of a linearly ordered set $(P,\leq)$ is called
	\emph{cofinal} if, for every $x\in P$, there exists $c\in C$
	such that $x\leq c$. The cofinality of $P$, denoted by
	$\operatorname{cf}(P)$, is the least cardinality of a cofinal
	subset of $P$.
	
	\item[ii)]  Dually, a subset $C\subseteq P$ is called \emph{coinitial} if,
	for every $x\in P$, there exists $c\in C$ such that $c\leq x$.
	Its least possible cardinality is called the coinitiality of $P$.
\end{enumerate} 
\end{discussion}

In what follows,  $\omega$ denotes the first infinite ordinal, identified with
the ordered set of nonnegative integers. We write $\omega_1$ and
$\omega_2$ for the first and second uncountable initial ordinals,
respectively.
Here, $\omega_1^*$ denotes the reverse ordering of the first
uncountable ordinal $\omega_1$. More precisely, $\omega_1^*$
is the set of ordinals $\alpha<\omega_1$, equipped with the
reverse order, i.e., $
\alpha\leq_{\omega_1^*}\beta
\quad\Longleftrightarrow\quad
\beta\leq\alpha
$
in the usual ordering of ordinals. Thus, $L=\omega_1^*+\omega_2$
is the ordered sum obtained by placing $\omega_1^*$ before
$\omega_2$.

\begin{example}  We present an non-noetherian example with $\pd_R(\widehat R^{\,\m})={\pd_R(\widetilde R)=3}$. 
\end{example}
\begin{proof}
Let
$
L:=\omega_1^*+\omega_2,$
$\Gamma:=\bigoplus_{\ell\in L}\mathbb Z e_\ell,
$
with the lexicographic ordering determined by the largest index
at which two elements differ. 
 For a nonzero element
$
\gamma=\sum_{\ell\in L}a_\ell e_\ell\in\Gamma,
$
we have $\gamma>0$ precisely when the coefficient at the largest
index in its support is positive.
The coinitiality of $\Gamma_{>0}$ is the least cardinality of a
subset $C\subseteq\Gamma_{>0}$ such that every positive element
of $\Gamma$ is bounded below by some element of $C$.
Similarly, the cofinality of $\Gamma$ is the least cardinality
of a cofinal subset.
Since the order on $L$ is the ordered sum
$\omega_1^*+\omega_2$, the two cardinal invariants are governed
by its two components. In particular, the $\omega_1^*$-component
is relevant to the coinitiality, whereas the $\omega_2$-component
is relevant to the cofinality.
Inside the Hahn field $H:=k((t^\Gamma))$, let
$$
F:=\{f\in H:|\supp(f)|<\aleph_2\},\qquad
R:=\{f\in F:v(f)\ge 0\}.
$$
The support restriction defines a subfield: sums, products and inverses preserve this cardinal bound. Hence $R$ is a valuation domain with fraction field $F$ and residue field $k$.

The valuation-domain dimension formula says that a noncyclic torsion-free uniserial module requiring exactly $\aleph_n$ generators has projective dimension $n+1$. This is the classical Osofsky--Small calculation; see also \cite[Theorem~14]{Simmons1987}  applied to torsion-free modules.
The maximal ideal $\m$ requires exactly $\aleph_1$ generators, by the coinitiality of $\Gamma_{>0}$. Thus
$
\pd_R()\m)=2.
$
There is no least positive value,  a well-known fact  implies that $\m^2=\m$.  Inductively, $\m^n=\m$. This implies that $$\widehat R^{\,\m}=\vpl R/ \m^n=\vpl R/ \m =k.$$ In view of 
$0\to\m\to R\to k\to 0$, 
and $\pd_R()\m)=2.$, we see
${\pd_R(\widehat R^{\,\m})=3},$

Cofinality $\aleph_2$ implies that $F$, as an $R$-module, requires exactly $\aleph_2$ generators. Therefore
$
\pd_R F=3.
$
Moreover, $R$ is $R$-adically incomplete. Indeed, writing $j_\alpha$ for the indices in the $\omega_2$-part of $L$, the partial sums
$
h_\beta=\sum_{\alpha<\beta<\omega_2}t^{e_{j_\alpha}} 
$
belong to $R$ and form an $R$-adic Cauchy net. Their Hahn-series limit has support of cardinality $\aleph_2$, so it does not belong to $F$.
For any domain, the quotient $\widetilde R/R$ is torsion-free and divisible, hence an $F$-vector space. Here it is nonzero, so
$
\widetilde R/R\cong F^{(\lambda)}$ where $ \lambda>0.$
Applying $\Ext$ to
$
0\to R\to\widetilde R\to F^{(\lambda)}\to 0
$
gives
${\pd_R(  \widetilde R)=3}.
$
\end{proof}

\begin{example}Let $k$ be any field, and take
$
R:=k[[x]][y]_{(x,y)},  \m:=(x,y)R.
$
Then	\begin{enumerate}
	\item[i)] $\widehat R^{\,\m}=k[[x,y]]$  and $\widetilde R\cong k[y]_{(y)}[[x]]$.
	
	\item[ii)]  $\pd_R(\widehat R^{\,\m})=\pd_R(\widetilde R)=2.$  
\end{enumerate}

\end{example}

\begin{proof}

 Let $Q:=\Frac(R)$. The ring $R$ is regular local of dimension two, so its global dimension is two. Thus every $R$-module has projective dimension at most two.
For each $f\in xk[[x]]$, the element $y-f(x)$ generates a prime ideal, since
$
R/(y-f(x))\cong k[[x]].
$
These give uncountably many distinct prime elements, even when $k$ is finite. Since $R$ is a UFD, a countable collection of denominators involves only countably many prime factors. Consequently, $Q$ cannot be countably generated over $R$.
Kaplansky's Theorem~1 (see \cite{Kaplansky1966}) states that, for a local domain,
$$
\pd_R (Q)\le 1 \quad\Longrightarrow\quad Q\text{ is countably generated over }R.
$$
Therefore $\pd_R (Q)=2$.

i) Set
$
S:=\widehat R^{\,(x)}:=\varprojlim_n R/x^nR\cong k[y]_{(y)}[[x]].
$
We have natural inclusions $R\subset S\subset k[[x,y]]$.
We first claim that every $R/rR$, for $r\ne 0$, is $x$-adically complete. By Weierstrass preparation and polynomial division, every nonzero $r\in R$ is associated to
$
x^aP(y),
$
where $P\in k[[x]][y]$ is monic and $P\bmod x=y^d$; $P=1$ is allowed. Indeed, for a polynomial numerator, polynomial division by its Weierstrass polynomial shows that its unit factor is also polynomial, hence a unit in $R$ \cite{HochsterStructure}.
Now $R/PR$ is finite over $k[[x]]$, hence $x$-adically complete. Also $R/x^aR$ is $x$-adically complete. The exact sequence
$
0\longrightarrow R/PR \xrightarrow{\;\cdot x^a\;} R/x^aPR
\longrightarrow R/x^aR \longrightarrow 0
$
and exactness of completion on finite modules prove the claim. The same completion theorem gives for any nonzero $r$ that
$
R/rR\xrightarrow{\;\sim\;}S/rS$.
It follows that multiplication by every nonzero $r\in R$ is bijective on $V=S/R$. Hence $V$ is a $Q$-vector space. Moreover, $V\ne 0$. For example,
$
h=\sum_{n\ge 1}x^ny^{n!}\in S\setminus R.
$
Indeed, $R\subset k((x))(y)$, whereas the arbitrarily large gaps between successive nonzero $y$-coefficients of $h$ rule out rationality: multiplying by any fixed nonzero polynomial leaves isolated nonzero coefficients in arbitrarily large degrees.
Because every $R/rR$ is $x$-adically complete, its quotient map from $R$ extends uniquely to $S$. These compatible maps give
$
S\to \widetilde R.
$
Restriction to the coordinates $r=x^n$ gives the inverse map $\widetilde R\to S$. Injectivity follows from
$\bigcap_n x^n(R/rR)=0$. Thus $\widetilde R\cong S$.

ii)   Recall that projective dimenstion of any module is bounded by global dimenstion, which two here. Also, $\pd_R(Q)=2$. It remains  to apply proposition \ref{43}, i.e., 
$2\leq \pd_R(Q) = \pd_R( \overset{\sim} R )\leq \pd_R(\widehat{R})\leq \gldim (R)=2.$
\end{proof}

\begin{example} Assuming $2^{\aleph_1}=\aleph_2$, one can construct a three-dimensional Noetherian regular local domain $R$ such that
$
 {\pd_R(\widehat R)=\pd_R(\widetilde R)=3.}
$
 
\end{example}

\begin{proof}
Let $\kappa:=\aleph_2$, and  define
$k:=\mathbb Q(u_\xi:\xi<\kappa)$, 
$T:=k[[x,y,z]]$. Also, set  $h:=\exp(x)=\sum_{j\geq0}\frac{x^j}{j!}.$
We regard $k[x,y,z]_{(x,y,z)}$ as a subring of $T$ and write
$A_0=k[x,y,z]_{(x,y,z)}.$
The cardinal assumption implies
$$
|k|=\kappa,\qquad |T|=|k|^{\aleph_0}
=\kappa^{\aleph_0}
=(2^{\aleph_1})^{\aleph_0}
=2^{\aleph_1\cdot\aleph_0}
=2^{\aleph_1}=\kappa.
$$

We construct a continuous increasing chain of local subrings
$$
A_0\subseteq A_1\subseteq\cdots\subseteq A_\alpha
\subseteq\cdots\subseteq T\qquad \forall\alpha<\kappa,
$$
with maximal ideal $\m_\alpha=A_\alpha\cap\fn$. The final ring will be
$
R=\bigcup_{\alpha<\kappa}A_\alpha.
$
We will construct a ring $R$ so that
$
k[x,y,z]_{(x,y,z)}\subseteq R\subseteq T
$
with these properties:
\begin{align}
	&IT\cap R=I &&\text{for every finitely generated ideal }I\subseteq R; \tag{1}\\
	&h\notin\Frac(R); \tag{2}\\
	&h\in R+rT &&\text{for every }0\ne r\in R. \tag{3}
\end{align}

\begin{claim}\label{c711}
Let $L:=\Frac(T)$. 
Then
$
\trdeg_k L=\kappa.
$
\end{claim}

\begin{PROOF}{\ref{c711}}We show
	$
	\mathcal E:=\{\exp(x)\}\cup
	\{\exp(u_\xi x):\xi<\kappa\}
	$
	is algebraically independent over $k(x,y,z)$.
 Choose finitely many distinct elements
$a_1,\ldots,a_m$ from $\{1\}\cup\{u_\xi:\xi<\kappa\}$. Suppose that
$P(\exp(a_1x),\ldots,\exp(a_mx))=0
$ for a nonzero polynomial $P$ over $k(x,y,z)$. On expanding, and collecting equal
exponents, we obtain a relation
$$(\ast):\quad\quad\quad
\sum_{j=1}^q f_j(x,y,z)\exp(\lambda_jx)=0,
$$
where the $\lambda_j$ are distinct elements of the additive subgroup of $k$
generated by $a_1,\ldots,a_m$, and each $f_j$ is a nonzero element of
$k(x,y,z)$. The distinctness of the monomials' exponent vectors does not by
itself guarantee distinct $\lambda_j$ for arbitrary rates; instead, one first
chooses the $a_i$ algebraically independent over $\mathbb Q$, which makes
distinct integer combinations distinct. This applies to any finite subset
of the displayed family.
Let $D:=\partial/\partial x$. For any $j$, apply to
$(\ast)$ the operator
$\prod_{i\ne j}(D-\lambda_i)^{N_i},$
where the positive integers $N_i$ are chosen large enough to eliminate the
terms with exponent $\lambda_i$ after the usual repeated-operator argument.
Equivalently, one may use induction on $q$: applying $D-\lambda_1$ gives
$\sum_{j=2}^q\bigl(Df_j+(\lambda_j-\lambda_1)f_j\bigr)
\exp(\lambda_jx)=0.$
For $c\ne0$, the operator $D+c$ is injective on $k(y,z)(x)$: a nonzero
rational function cannot satisfy $f'+cf=0$ (for example, comparison of the
orders at infinity gives a contradiction). Thus the coefficients remain
nonzero, and induction contradicts the existence of a nontrivial relation.
This proves the asserted algebraic independence.\end{PROOF}

Consequently $\trdeg_k L\geq\kappa$. Since $|T|=\kappa$, we also have
$\trdeg_k(  L)\leq\kappa$, proving the equality.
We will repeatedly use the following consequence. If $F$ is a subfield of
$L$ with $\trdeg_k (F)<\kappa$, then there are elements of $T$ algebraically
independent over $F$, in any prescribed finite number. The same remains true
after adjoining finitely many prescribed elements of $T$ to $F$, provided the
resulting transcendence degree is still less than $\kappa$.
Suppose that $A$ is a local subring of $T$ containing $A_0$, with maximal ideal
$A\cap\fn$, and let $F:=\Frac(A)$. Assume
$\trdeg_k ( F)<\kappa$ and $ h${ is transcendental over }$F.$
For a subring $B$ of $T$, write
$B^{\mathrm{loc}}:=B_{B\setminus(B\cap\fn)}.$
This is a local subring of $T$ whose maximal ideal is its intersection with
$\fn$. Localization does not change the fraction field. Let us bring two procedures that we need.

\smallskip
\noindent\emph{Procedure I (approximation modulo a nonzero element).}
Let $0\ne r\in A$. Choose $v\in T$ transcendental over $F(h)$, and set
$$
b:=h+rv,\qquad B:=A[b],\qquad A':=B^{\mathrm{loc}}.
$$
Then $h-b=-rv\in rT$, so
$h\in b+rT\subseteq A'+rT.
$ Moreover, $h$ is transcendental over $F(b)$. Indeed, $b$ is transcendental
over $F(h)$: otherwise $v=(b-h)/r$ would be algebraic over $F(h)$,
contrary to its choice. If $h$ were algebraic over $F(b)$, then
$\operatorname{trdeg}_F F(h,b)=1$, whereas the transcendence of $b$ over
$F(h)$ and of $h$ over $F$ gives
$\operatorname{trdeg}_F F(h,b)=2$. This is impossible. Hence $h$ remains
transcendental over $\Frac(A')=F(b)$.

\smallskip
\noindent\emph{Procedure II (solving one ideal-membership requirement).}
Suppose that $a_1,\ldots,a_n,c\in A$, that not all $a_i$ are zero, and that
$
c\in(a_1,\ldots,a_n)T.
$
After reordering, assume $a_1\ne0$. Choose $s_1,\ldots,s_n\in T$ with
$c=\sum_{i=1}^n a_is_i.$
If $n=1$, put $b_1=s_1$ and $A':=A[b_1]^{\mathrm{loc}}$. Since
$a_1b_1=c$, the desired membership holds. Also $b_1=c/a_1\in\Frac(A)$,
so the fraction field does not change.
Suppose now that $n\geq2$. Choose
$v_2,\ldots,v_n\in T$ algebraically independent over
$F(h,s_1,\ldots,s_n)$, and define
$$
b_1:=s_1-\sum_{i=2}^n a_iv_i,
\qquad
b_i:=s_i+a_1v_i\quad 2\leq i\leq n .
$$
Using $c=\sum_{i=1}^n a_is_i$, we calculate
\begin{align*}
	\sum_{i=1}^n a_ib_i
	&=a_1s_1-a_1\sum_{i=2}^n a_iv_i
	+\sum_{i=2}^n a_is_i
	+\sum_{i=2}^n a_1a_iv_i
 =\sum_{i=1}^n a_is_i=c.
\end{align*}
Thus $c\in(a_1,\ldots,a_n)A'$, where
$A'=A[b_1,\ldots,b_n]^{\mathrm{loc}}$.
Furthermore, the identity
$a_1b_1=c-\sum_{i=2}^n a_ib_i$
and $a_1\ne0$ show that
$F(b_1,\ldots,b_n)=F(b_2,\ldots,b_n).$
The elements $b_2,\ldots,b_n$ are algebraically independent over
$F(h,s_1,\ldots,s_n)$, since the change of variables
$b_i=s_i+a_1v_i$ is invertible over that field. In particular, they are
algebraically independent over $F(h)$. It follows that $h$ is transcendental
over $F(b_2,\ldots,b_n)$: otherwise the transcendence degree of
$F(h,b_2,\ldots,b_n)$ over $F$ would be at most $n-1$, whereas it is
$n$. By $F(b_1,\ldots,b_n)=F(b_2,\ldots,b_n)$, $h$ is transcendental over
$\Frac(A')$ as well.
In either procedure the fraction field has transcendence degree over $k$
strictly less than $\kappa$. Indeed, at a successor stage we adjoin only
finitely many elements, and $\kappa=\aleph_2$ is uncountable and regular.

\begin{claim}\label{c712}
The above procedures are iterated.
\end{claim}

\begin{PROOF}{\ref{c712}}
At stage $\alpha$,
there are at most $\kappa$ requirements of each of the following forms:
\begin{itemize} 
	\item for every $0\ne r\in A_\alpha$, perform Procedure I for $r$;
	\item for every finite tuple $a_1,\ldots,a_n,c\in A_\alpha$ such that
	$c\in(a_1,\ldots,a_n)T$, perform Procedure II for that tuple.
\end{itemize}
There are at most $\kappa$ such requirements because $A_\alpha\subseteq T$
and $|T|=\kappa$. List the requirements at stage $\alpha$ as
$(\mathcal R_{\alpha,\xi})_{\xi<\kappa}$, allowing repetitions or dummy
requirements if necessary.
There is a bijection
$$
\gamma\longmapsto(\alpha_\gamma,\xi_\gamma)
\quad\text{from }\kappa\text{ onto }\kappa\times\kappa
$$
such that $\alpha_\gamma\leq\gamma$ for every $\gamma<\kappa$. To see this,
well-order $\kappa\times\kappa$ by increasing value of
$\max\{\alpha,\xi\}$, and within each level choose any well-ordering. Since
$\kappa$ is regular and every proper initial square has cardinality less than
$\kappa$, the resulting enumeration has the stated property (one may choose
an enumeration with order type $\kappa$).
Inductively, having constructed $A_\gamma$, handle the requirement
$\mathcal R_{\alpha_\gamma,\xi_\gamma}$, which is available because
$\alpha_\gamma\leq\gamma$ and the chain is increasing. Apply the appropriate
procedure to $A_\gamma$ (the data from $A_{\alpha_\gamma}$ remain in
$A_\gamma$), and then localize as described above to obtain $A_{\gamma+1}$.
At a limit ordinal $\lambda<\kappa$, set
$A_\lambda=\bigcup_{\alpha<\lambda}A_\alpha.$
The union is a local ring with maximal ideal $A_\lambda\cap\fn$. The
transcendence condition is preserved: if $h$ became algebraic over
$\Frac(A_\lambda)$, an algebraic equation would involve only finitely many
coefficients, all of which would occur in some earlier $A_\alpha$, a
contradiction.
At every stage $\gamma<\kappa$, fewer than $\kappa$ elements have been
adjoined over $A_0$ (at most finitely many at each earlier successor stage).
Thus $\trdeg_k\Frac(A_\gamma)<\kappa$, and the choices required in the two
procedures remain possible. This completes the recursion.
\end{PROOF}

Let $J$ be any ideal of $R$. Since $T$ is Noetherian, finitely many elements $a_1,\ldots,a_n\in J$ generate $JT$. By (1) we deduce 
$J\subseteq(a_1,\ldots,a_n)T\cap R
=(a_1,\ldots,a_n)R\subseteq J.
$
Thus, $J$ is finitely generated, and  so  $R$ is Noetherian.
Again by (1), its maximal ideal is $(x,y,z)R$, and
$
R/(x,y,z)^nR\cong T/(x,y,z)^nT.
$
Surjectivity follows because $R$ contains $k[x,y,z]$. Hence
$
 \widehat R\cong\varprojlim_q R/\m^q
\cong\varprojlim_q T/\fn^q=T.
$
Since $T$ is regular local of dimension three, so is $R$; therefore $\gldim R=3$.
By (3), choose $b_r\in R$ with $h-b_r\in rT$. Property (1) makes the residues $b_r+rR$ compatible, so they define an element of $\widetilde R$.
If this element came from $a\in R$, we would have $h-a\in rT$ for every nonzero $r\in R$. Taking $r=x^n$ gives
 $h-a\in\bigcap_n x^nT=0,
$
contrary to (2). Thus $R\ne\widetilde R$.

We have
$
|R|=|R/\m|=\aleph_2.
$
Furthermore, $Q=\Frac(R)$ requires exactly $\aleph_2$ generators over $R$. Indeed, the elements
$
x+ay$ with $ a\in k$.
generate distinct height-one prime ideals: this follows by extending to $T$ and using (1). Any denominator belongs to only finitely many of these primes. Hence fewer than $\aleph_2$ denominators cannot generate $Q$.
Osofsky's quotient-field theorem states that, for a regular local ring with $|R|=|R/\m|$, if $Q$ requires exactly $\aleph_n$ generators, then
$
\pd_R (Q)=\min\{\dim R,n+1\},
$
see \cite{Osofsky1968}. Thus $\pd_R (Q)=3$.
Now apply  Proposition~ \ref{43}:
$
3=\pd_R (Q)=\pd_R(\widetilde R)\le\pd_R(\widehat R)\le\gldim( R)=3.
$
\end{proof}

\subsection{A computation of $\widehat{R}\otimes_R \widehat{R}$}

\begin{proposition}\label{46}
Let $(R,\fm)$ be a local   integral domain. Let $u$ be  dimension of the $Q$-vector space of $\Ext^1_R (Q,R)$. If $R = \overset{\sim}R$, then $\overset{\sim}R \otimes _R\overset{\sim}R = R$.
Suppose $\overset{\sim}R \neq R$, then
$$\otimes_n\overset{\sim}R:=\overset{\sim}R \otimes_R \ldots\otimes_R\overset{\sim}R = \overset{\sim}R \oplus Q^{(b_n)},\qquad
b_1:=0,\quad b_{n+1}:=uv+vb_n,\quad
b_n:=u\sum_{j=1}^{n-1}v^j.$$
\end{proposition}
\begin{proof}
There exists an sequence
 $0  \to R  \to \overset{\sim}R  \to Q^{\oplus u}\  \to 0.$ 
Since $\overset{\sim}R$ is flat as $R$-module
$$0 \longrightarrow \overset{\sim}R \longrightarrow \overset{\sim}R \otimes_R \overset{\sim}R \longrightarrow Q^u \otimes_R \overset{\sim}R \longrightarrow 0,$$is exact.
Recall that $\overset{\sim}R$ is cotorsion  and $\overset{\sim}R \otimes_R Q  $ is a $Q$-vector space.
Then 
 $$\begin{array}{ll}
\Ext^1_R ((\overset{\sim}R \otimes_R Q)^{\oplus u} , \overset{\sim}R)  = \Ext^1_R (\oplus Q , \overset{\sim}R) =\prod   \Ext^1_R ( Q , \overset{\sim}R) =0.\\
 \end{array}$$
 
  We apply this to see the exact sequence
 $0 \to\overset{\sim}R  \to \overset{\sim}R \otimes_R \overset{\sim}R  \to (\overset{\sim}R \otimes _RQ)^{\oplus u} \lo 0,$ 
is split, i.e.,
$$ \overset{\sim}R \otimes_R \overset{\sim}R \cong  \overset{\sim}R \oplus  (\overset{\sim}R \otimes_R Q)^{\oplus u}\quad(\ast)$$
Now we compute $(\overset{\sim}R \otimes_R Q)$.
Apply the exact functor $Q \otimes_R -$ to
$ 0 \to R \to \overset{\sim}R \to \oplus Q \to 0$ and deduce the exact sequence
$$0 \longrightarrow Q \longrightarrow Q \otimes_R\overset{\sim}R \longrightarrow \oplus( Q \otimes_R Q) =\oplus Q\longrightarrow 0.$$
Since $Q$ is injective the sequence splits, i.e.,  $Q \otimes_R\overset{\sim}R = \oplus_u( Q \otimes_R Q)\oplus Q=Q^{u(u+1)}$  We put this in $(\ast)$,
and deduce that $ \overset{\sim}R \otimes_R \overset{\sim}R \cong  \overset{\sim}R \oplus ( Q)^{\oplus u( u+1)}$. Suppose, inductively that  $\otimes_n  \overset{\sim}R = \overset{\sim}R \oplus Q^{(b_{n})}.$ Tensor this with $\overset{\sim}R$ yields the claim.
\end{proof}

\begin{corollary} 
	Suppose $R$ is a 1-dimensional local domain which is not complete with respect to $\fm$-adic topology. Let $u:=\dim_Q(\Ext^1_R(Q,R))$. Then
	
	\begin{enumerate}
		\item[i)] $\widehat{R}\otimes_R\widehat{R}=\widehat{R}\oplus Q^{u(u+1)}$
		
		\item[ii)] $\pd_R(\widehat{R}\otimes_R\widehat{R})=1$.
			\end{enumerate}
\end{corollary}

\section{Extensions of adic completions}

Suppose $(R,\frak{m})$ is a 1-dimensional regular ring. It is  proved in \cite{wag} that
 $\Ext_R^1 (\widehat{R}, R) = E_R(R/\fm) \oplus Q^t,$ 
where $t<\infty$ or $t$ is uncountable.
Here, we extend it to the Gorenstein case.

\begin{proposition}\label{18}
	Let $(R,\frak{m})$ be a 1-dimensional non-complete Gorenstein local integral domain such that $\widehat{R}$ is reduce. Then
 $\Ext_R^1 (\widehat{R}, R) = E(R/\frak{m}) \oplus Q^t.$ 
	The number $t$ can be determined. For example if $p$
	being a prime numb, then
	$$
	t	=\left\{\begin{array}{ll}
	(p^i -1)^2 -1&\mbox{if }  t<\infty \\
	uncountable	&\mbox{if }   t=\infty.
	\end{array}\right.
	$$
	
\end{proposition}

\begin{proof}
	Matlis    proved the following sequence
 $0 \longrightarrow R \longrightarrow \overset{\sim}{R} \longrightarrow \Ext_R^1 (Q,R) \longrightarrow 0 $ 
	is exact.
	It is easy to see $\Ext_R^1 (Q,R)$ is a $Q$-vector space, then there exists $t \geq 0$ such that $\Ext_R^1 (Q,R) \simeq Q^t$.
	Jensen \cite{j} proved in theorem 1 that any infinite cardinal can happen in (*) among finite $ t $ the $p^i -1$ when $p$ is prime and $i > 0$ allowed. 	Since $\dim(R)=1$ the $\frak{m}$-adic topology is the same as $R$-topology, i.e.
 $\vpl_{n\in \mathbb{N}} R/\frak{m}^{n} := \widehat{R} = \overset{\sim}{R} := \vpl_{r\in R} R/rR.$ In summary, there exists an exact sequence
 $0 \to R \to\widehat{R} \to \oplus Q \to 0.$ 
	Applying $\Hom_R (-,R)$ to it:
	$$\Hom_R (\widehat{R},R) \longrightarrow \Hom_R (R,R) \longrightarrow \Ext_R^1 (\oplus Q,R) \longrightarrow \Ext_R^1 (\widehat{R} ,R) \longrightarrow \Ext_R^1 (R,R)=0.$$
	Also, $\Ext_R^1(\oplus Q,R) = \sqcap \Ext_R^1(\oplus Q,R)  = \sqcap (\oplus Q)$
	Now there is a vector space over $Q$ of dimension
	$t_1:=t^2$ if $t$ is finite and $t_1:=2^t$ if $t$ is infinite. According to Matlis (see Fact \ref{hom}) $\Hom_R (\widehat{R},R)=0$.
	So
 $ 0  \to R  \to \oplus Q  \to \Ext_R^1 (\widehat{R},R)  \to 0.$ 
	Since $R$ is $1$-dimensional Gorenstein, we know $\id_R(R)=1$ and also $Q$ is injective.
	From these, we conclude $\Ext_R^1 (\widehat{R},R)$ is injective. By Matlis decomposition,
 $\Ext_R^1 (R,R) = E(R/\fm)^\alpha \oplus E(R)^\beta,$ 
	where $\alpha,\beta$ are some cardinals.
	By the proof of page 172 of \cite{wag} $\alpha=1$.
	\footnote{Let us show the possibility $\beta=0$. If $R$ is $Q$-ring, then $t=1$ and $t_{1}=1$.} Tensoring with flat module $Q$ yields that
 $0  \to Q  \to \oplus Q  \to \Ext_R^1 ( \widehat{R},R) \otimes Q  \to 0.$ 
	Also,
	$$\begin{array}{ll}
	\Ext_R^1(\widehat{R},R) \otimes Q  =(Q \otimes E_R(R/\frak{m})) \oplus ( Q \otimes (\oplus Q)^\beta ) 
	 = 0 \oplus ( \oplus Q^\beta )	  =Q^\beta.
	\end{array}$$
	Put this in the above sequence and compute the dimension, lead us to
 $0  \to Q  \to\underset{t_1}\oplus Q  \to\underset{\beta}\oplus Q \to 0.$ 
	Now, $\beta =t_1 -1 $
	and then $\beta = ((p^i -1)-1)^2$ if $t < \infty $, and $2^t$ if $t=\infty$. But $2^t$ is uncountable.
\end{proof}

Recall from \cite{wag} that $\Ext^{i}_{R}(\widehat{R},R)=0 , \forall i\geq2$, where $R$ is a 1-dimensional regular ring which is not complete.

\begin{corollary}\label{20}
	Let $(R,\frak{m})$ be a 1-dimensional integral domain which is not complete. Then $\pd_{R} (\widehat{R})=1$.
	In particular, $\Ext_{R}^{i}(\widehat{R},R)=0$   for all $i\geq2$.
\end{corollary}

\begin{proof}
	Since $\dim(R)=1$, we know $\widehat{R}=\overset{\sim}R$. Also, there is an exact sequence
	$$\eta: 0 \longrightarrow R \longrightarrow \widehat{R} \longrightarrow \Ext^{1}_R(Q,R) \longrightarrow 0$$
	and the natural isomorphism $\Ext^{1}_R(Q,R) \cong \oplus Q$.
	Let $i \geq 2$, and apply $\Hom_
	R (-,R)$ to $\eta$ yields that
	$$0=\Ext _{R}^{i-1} (R,R) \longrightarrow \Ext _{R}^{i} (\oplus Q, R) \longrightarrow \Ext _{R}^{i} (\widehat{R},R) \longrightarrow \Ext _{R}^{i} (R,R) =0. $$
Since  $\pd(Q)=1$, and $i \geq 2$  we have $\Ext_R^i(Q,R)=0$. By plugging this in the previous sequence yields that  $\Ext_{R}^{i}(\widehat{R},R)\cong\Ext _{R}^{i} (\oplus Q, R)\cong\prod\Ext_R^i(Q,R)=0,$ 
as claimed.\end{proof}
Against to projective dimension, 
it is easy to find $\id_R (\widehat{R})$:

\begin{corollary}\label{22}  Let $(R,\frak{m})$  be  a local ring which is not a field.   $\id_R (\widehat{R}) < \infty$.
 iff $R$ is Gorenstein.
\end{corollary}
\begin{proof}
	Suppose $t:=\id_R  (\widehat{R}) < \infty$. Then $\Ext^i_{R} (- , \widehat{R})=0$ for every $i >t$.  According to  Cartan-Eillenberg \cite[VI.4.1.3]{ce} we know
	$$\begin{array}{ll}
0	&=\Ext_{R}^i (R/\fm , \widehat{R}) =\Ext_{\widehat{R}}^i (R/\fm \otimes \widehat{R}, \widehat{R})\\
	&=\Ext_{\widehat{R}}^i ( \widehat{R}/ \fm \widehat{R} , \widehat{R})=\Ext_{\widehat{R}}^i ( \widehat{R}/ \fm_{ \widehat{R}} , \widehat{R}).
	\end{array}$$
 In other words, $\widehat{R}$ is Gorenstein and   $R$ is as well.
	Conversely, suppose $R$ is Gorenstein. Then
	$$ 0 \longrightarrow R \longrightarrow \oplus R_{x_{i}} \longrightarrow \ldots  \longrightarrow R_{x_{1}\ldots x_{d}}\longrightarrow 0$$
	is a flat resolution of $H^{d}_{\frak{m}} (R)$.
	Dualizing by $\Hom_{R}(-, E(R/\frak{m}))$ shows that
	$$ 0 \longrightarrow E(R/\frak{m})^\vee \longrightarrow R_{x_{1}\cdots x_{d}}^\vee \longrightarrow \ldots \longrightarrow R^\vee\longrightarrow 0$$
	is exact. Thus, $\id_{R}(\widehat{R}) < \infty$.
\end{proof}

\section{Matlis' decomposability problem} 

Let $(R,\fm)$ be a complete local domain of dimension one, and let $S$ be torsion-free and of rank one. It was a conjecture of Matlis that $Q/R\otimes_R S$ is indecomposable. 
In \cite{d} he was successful to show that it is indeed indecomposable.

\begin{discussion}\label{dis}Adopt the above notation. Then $H^1_\fm(S)\cong Q/S.$
\end{discussion}

\begin{proof}Apply $- \otimes_R S$ to  $ 0 \longrightarrow R \longrightarrow Q \longrightarrow Q/R \longrightarrow 0$
	gives us the following diagram:
	
	$$
	\begin{CD}
0=\Tor^R_1(Q, S)@>>>	\Tor^R_1(Q/R, S)@>>>   R\otimes_R S @>>> Q\otimes_R S@>>> Q/R\otimes_R S@>>>0\\
	@. @AAA\cong @AAA = @AAA   f @AAA \\
@.	0@>>> S @>>> Q @>>>Q/S@>>>0 \\
	\end{CD}$$On the one hand $\Tor^R_1(Q/R, S)\subseteq  R\otimes_R S =S$ is torsion-free. On the other hand any higher tor is torsion. So, $\Tor^R_1(Q/R, S)=0$. By 5-lemma,  $H^1_\fm(S)=H^1_\fm(R)\otimes_R S\cong Q/R \otimes_R S\stackrel{f}\cong Q/S.$
\end{proof}

 Let us extend Matlis' result to higher rank, via a modern proof.

\begin{theorem}
Let $(R,\fm)$ be a complete local domain of dimension one, and let $S$ be indecomposable torsion-free and of finite rank. Then 
$Q/R\otimes_R S$ is indecomposable provided it is nonzero.
\end{theorem}

\begin{proof}
	In view of \cite[Cor 3]{matcanada}, there is $t$ and a finitely generated module $M$ such that $S=Q^t\oplus M$. Since $S$ is indecomposable, either $S=Q$ or $S=M$.
	In the first case,  $0=H^1_\fm(S)
	=H^1_\fm(R)\otimes_RS\cong Q/R\otimes_R S.$  So, we are done. In the second case, we recall that a finitely generated module is indecomposable if its Matlis dual is as well.  Now, we use Grothendieck's local duality to see $H^1_\fm(M)^\vee=\Hom_R(M,\omega_R)$. As the ring is 1-dimensional, and $M$ is torsion-free, it is maximal Cohen-Macaulay. Now, we use \cite[3.3.10(d)]{BH} to see
 $M\cong \Hom_R(\Hom_R(M,\omega_R),\omega_R),$  to get the desired claim.
\end{proof}
 We leave to the reader to use some Mayer-Vietoris and present converse of above result.


\end{document}